\documentclass [11pt,a4paper]{amsart}

\usepackage{amssymb} 
\usepackage{amsfonts} 
\usepackage{amsmath}
\usepackage{amsthm} 
\usepackage{tikz}
\usepackage{enumerate,pb-diagram,pb-xy}

\usepackage[utf8]{inputenc}
\usepackage[all]{xy}
\SelectTips{cm}{}

\usepackage[pagebackref,
    ,pdfborder={0 0 0}
    ,urlcolor=black,a4paper,hypertexnames=false]{hyperref}
\hypersetup{pdfauthor={Roberto Frigerio, Clara L\"oh, Cristina Pagliantini, Roman Sauer},pdftitle=Integral foliated simplicial volume of aspherical manifolds}

\usepackage{caption}

\DeclareMathOperator{\vol}{vol}
\DeclareMathOperator{\inc}{inc}

\DeclareMathOperator{\NF}{\mathit{NF}}
\DeclareMathOperator{\im}{im}

\newtheorem{lemma}{Lemma}[section]
\newtheorem{thm}[lemma]{Theorem}
\newtheorem{prop}[lemma]{Proposition}
\newtheorem{cor}[lemma]{Corollary}

\newtheorem*{thm*}{Theorem}

\theoremstyle{definition}
\newtheorem{defn}[lemma]{Definition}

\newtheorem{quest}[lemma]{Question}

\newtheorem{rem}[lemma]{Remark}

\newcommand{\matN}{\ensuremath {\mathbb{N}}}
\newcommand{\R} {\ensuremath {\mathbb{R}}}
\newcommand{\Q} {\ensuremath {\mathbb{Q}}}
\newcommand{\Z} {\ensuremath {\mathbb{Z}}}

\newcommand{\matH} {\ensuremath {\mathbb{H}}}

\newcommand{\vare} {\ensuremath{\varepsilon}}
\newcommand{\vola} {{\rm vol}_{\rm alg}}

\newcommand{\str} {\ensuremath {{\rm str}}}

\newcommand {\tors}{\operatorname{tors}}
\DeclareMathOperator{\rk}{rk}

 \makeatletter
 \newcommand\norm{\bBigg@{0.8}}

 \makeatother
 \newcommand{\inparens}[2][flex]{\csname #1l\endcsname(#2%
                                 \csname #1r\endcsname)\mathclose{}}
 \newcommand{\inangles}[2][flex]{\csname #1l\endcsname\langle#2%
                                 \csname #1r\endcsname\rangle\mathclose{}} 
 \newcommand{\innorm}[2][flex]{\csname #1l\endcsname|#2%
                                 \csname #1r\endcsname|\mathclose{}}
 \newcommand{\indnorm}[2][flex]{\csname #1l\endcsname\|#2%
                                 \csname #1r\endcsname\|\mathclose{}}
 \newcommand{\indnorml}[4][flex]{\csname #1l\endcsname\|#2%
                                 \csname #1r\endcsname\|_{#3}^{#4}\mathclose{}}

\newcommand{\sv}[2][flex]{\indnorm[#1]{#2}}%
\newcommand{\isv}[2][norm]{\indnorml[#1]{#2}{\Z}{}}
\newcommand{\pfcl}[2][flex]{\csname #1l\endcsname[#2%
                            \csname #1r\endcsname]}
\newcommand{\ifsv}[2][norm]{\csname #1l\endcsname\bracevert\!#2\!%
                            \csname #1r\endcsname\bracevert}
\newcommand{\stisv}[2][flex]{\indnorml[#1]{#2}{\Z}{\infty}}

\newcommand{\pfc}[1]{%
  \widehat{#1}}

\newcommand{\symmdiff}{%
  \mathbin{\bigtriangleup}}

\newcommand{\linfz}[1]{%
  L^\infty(#1,\Z)}

\DeclareMathOperator{\EMD*}{EMD^*}
\DeclareMathOperator{\MD}{MD}
\DeclareMathOperator{\Aut}{Aut}
\DeclareMathOperator{\map}{map}
\DeclareMathOperator{\id}{id}

\newcommand{\actson}{\curvearrowright}

\def\geq{\geqslant}
\def\leq{\leqslant}

\title{Integral foliated simplicial volume of aspherical manifolds}

\author[R.~Frigerio]{Roberto Frigerio}
\address{Dipartimento di Matematica \\
Universit\`a di Pisa \\
56127~Pisa,}

\email{frigerio@dm.unipi.it}

\author[C.~L\"oh]{Clara L\"oh}
\address{Fakult\"at f\"ur Mathematik\\
         Universit\"at Regensburg\\
         93040~Regensburg\\
         }
\email{clara.loeh@mathematik.uni-r.de}

\author[C. Pagliantini]{Cristina Pagliantini}
\address{Department Mathematik\\
         ETH Zentrum\\
         8092~Z\"urich\\
        }
\email{cristina.pagliantini@math.ethz.ch}        
        
\author[R. Sauer]{Roman Sauer}

\address{Karlsruhe Institute of Technology\\
76131~Karlsruhe\\}
\email{roman.sauer@kit.edu}

\thanks{}



\thanks{}

\begin{document}

\begin{abstract}
  Simplicial volumes measure the complexity of fundamental cycles of
  manifolds. In this article, we consider the relation between
  simplicial volume and two of its variants --  
  the stable  integral simplicial volume and the integral foliated simplicial volume. 
  The definition of the latter depends on a choice of a measure preserving action 
  of the fundamental group  on a probability space. 
  
  We show that integral foliated simplicial volume is
  monotone with respect to weak containment of measure preserving actions and yields upper 
  bounds on (integral) homology growth. 
  
  Using ergodic theory we prove that simplicial volume, integral foliated
  simplicial volume and stable integral simplicial volume coincide 
  for closed hyperbolic $3$-manifolds and closed aspherical manifolds with 
  amenable residually finite fundamental group (being equal to zero in the latter case).
  
  However, we show that integral foliated simplicial volume
  and the classical simplicial volume do \emph{not} coincide for
  hyperbolic manifolds of dimension at least~$4$.
  
\end{abstract}
\maketitle

\section{Introduction}

Simplicial volume is a homotopy invariant of manifolds, measuring the
complexity of singular fundamental cycles with
$\R$-coefficients. The
\emph{simplicial volume} of an oriented closed connected
$n$-manifold~$M$ is defined as
\[ \sv M := \inf \bigl\{ |c|_1 
                 \bigm| \text{$c \in C_n(M;\R)$, $\partial c = 0$, 
                            $[c] = [M]_\R$}\bigr\} \in \R_{\geq 0},
\]
where $|c|_1$ denotes the $\ell^1$-norm of the singular chain~$c$ with
respect to the basis of all singular $n$-simplices in~$M$. Despite its
topological definition, simplicial volume carries geometric
information and allows for a rich interplay between topological and
geometric properties of manifolds~\cite{vbc,mapsimvol}.

A long-standing purely topological problem on simplicial volume was
formulated by Gromov~\cite[p.~232]{gromovasymptotic}\cite[3.1.~(e) on
  p.~769]{gromov-cycles}:

\begin{quest}\label{Gromov-quest}
	Let $M$ be an
	oriented closed connected aspherical manifold. Does $\sv M = 0$ 
	imply $\chi (M) = 0$? Does $\sv M = 0$ imply the vanishing of $L^2$-Betti numbers 
	of~$M$?
\end{quest}

 A possible strategy to answer Question~\ref{Gromov-quest} in the affirmative is to replace
simplicial volume by a suitable integral approximation, and then to use
a Poincar\'e duality argument to bound ($L^2$)-Betti numbers in terms
of this integral approximation. Finally, one should relate the integral approximation 
to the simplicial volume on aspherical manifolds. 

One instance of such an integral approximation is stable integral
simplicial volume. The \emph{stable integral simplicial volume} of an
oriented closed connected manifold~$M$ is defined as
\[ \stisv M := \inf \Bigl\{ \frac1d \cdot \isv {\overline M} 
                    \Bigm| \text{$d \in \matN$ and $\overline M \rightarrow M$ 
                           is a $d$-sheeted covering}
                    \Bigr\} 
   ,
\]
where the integral simplicial volume~$\isv{\overline M}$ is
defined like the ordinary simplicial volume but using $\Z$-fundamental
cycles.

Another instance of this strategy is integral foliated simplicial
volume. A definition of integral foliated simplicial volume and a
corresponding $L^2$-Betti number estimate was suggested by
Gromov~\cite[p.~305f]{gromov} and confirmed by
Schmidt~\cite{mschmidt}. \emph{Integral foliated
simplicial volume}~$\ifsv M = \inf_\alpha \ifsv M ^\alpha$ is defined
in terms of fundamental cycles with twisted coefficients in~$\linfz
X$, where $\alpha = \pi_1(M) \actson X$ is a probability space with a
measure preserving action of the fundamental group of~$M$ (see
Section~\ref{sec:ifsv} for the exact definition).

For all oriented closed connected manifolds~$M$ these simplicial
volumes fit into the sandwich~\cite{loehpagliantini}
\[ \sv M \leq \ifsv M \leq \stisv M.
\]

This leads to the following fundamental problems~\cite[Question~1.7,
  Question~7.2, Question~4.21]{loehpagliantini}:

\begin{quest}\label{first:quest}
  Do simplicial volume and integral foliated simplicial volume
  coincide for aspherical manifolds? 
\end{quest}

\begin{quest}
 Do integral foliated simplicial
  volume and stable integral simplicial volume coincide for aspherical
  manifolds with residually finite fundamental group?
\end{quest}

\begin{quest}\label{qu: dependency on prob space}
  How does integral foliated simplicial volume depend on the action on the 
  probability space? If it does, is it an interesting dynamical invariant 
  of the action? Which role is played by the Bernoulli shift? 
\end{quest}

In this article, we contribute (partial) solutions to these questions.
On the one hand, we show that integral foliated simplicial volume is
compatible with weak containment of measure preserving actions. In
combination with results from ergodic theory this shows that integral
foliated simplicial volume and stable integral simplicial volume
coincide in various cases, e.g., for amenable residually finite or
free fundamental group.  Moreover, we give a geometric proof of the
fact that stable integral simplicial volume, integral foliated
simplicial volume, and simplicial volume all are zero for aspherical
manifolds with amenable fundamental group.

On the other hand, we show that integral foliated simplicial
volume and classical simplicial volume do \emph{not} coincide for
hyperbolic manifolds of dimension at least~$4$. Since hyperbolic
manifolds are aspherical, this answers Question~\ref{first:quest} in
the negative (even when restricting to manifolds with residually
finite fundamental group).

We now describe these results in more detail.

\subsection{Integral foliated simplicial volume and weak containment}

The measure preserving actions used to define integral foliated simplicial
volume are organized into a hierarchy by means of weak containment
(see Section~\ref{subsec:wcprelim} for the definitions). Integral 
foliated simplicial volume is compatible with this hierarchy in the 
following sense (Theorem~\ref{thm:wc}):

\begin{thm}[monotonicity of integral foliated simplicial volume]\label{intro-thm:wc}
  Let $M$ be an oriented closed connected manifold with 
  fundamental group~$\Gamma$, and let $\alpha = \Gamma
  \actson (X,\mu)$ and $\beta= \Gamma
  \actson (Y,\nu)$ be free non-atomic standard
  $\Gamma$-spaces with~$\alpha \prec \beta$ (i.e., $\alpha$ is weakly contained in $\beta$). Then
  \[ \ifsv M^\beta \leq \ifsv M^\alpha. 
  \]
\end{thm}

In combination with results from ergodic theory, we obtain the
following consequences for integral foliated simplicial volume:
Bernoulli shift spaces give the maximal (whence bad) value among free
measure preserving actions (Corollary~\ref{cor:maxbernoulli}). If the
fundamental group satisfies the universality property~$\EMD*$
(Definition~\ref{def:EMD*}) from ergodic theory, then integral
foliated simplicial volume and stable integral simplicial volume
coincide (Corollary~\ref{cor:ifsvpfc}). For instance, this applies to
free fundamental groups (Corollary~\ref{cor:free}) and to residually
finite amenable fundamental groups (Corollary~\ref{cor:amenable}).


\subsection{Integral foliated simplicial volume and bounds on homology}

A \emph{chain} in a group $\Gamma$ is a descending 
sequence $\Gamma=\Gamma_0>\Gamma_1>\Gamma_2>\ldots$ of finite index subgroups. 
We associate to a chain a measure preserving action on the coset tree, i.e., the inverse limit of the $\Gamma/\Gamma_i$ (see Subsection~\ref{subsec: actions}). 
We denote the torsion subgroup of a finitely generated abelian group~$A$ by~$\tors A$; it is 
a finite abelian group. 

\begin{thm}[homology bounds]\label{intro-thm:homology bound}
  Let $n \in \matN$. 
  Let $M$ be an oriented closed connected $n$-manifold with fundamental group~$\Gamma$, 
  let $(\Gamma_i)_i$ be a chain of~$\Gamma$, and let $M_i\to M$ be the finite covering 
  associated to $\Gamma_i$. Let $\alpha$ be the standard $\Gamma$-action on the 
  coset tree of $(\Gamma_i)_i$. Then for every~$k\in\matN$ and for every principal ideal domain~$R$ 
  we have (where $\rk_R$ denotes the $R$-dimension of the free part of finitely generated 
  $R$-modules)
  \begin{align*}
    \limsup_{i\to\infty} \frac{\log |\tors H_k(M_i;\Z)|}{[\Gamma:\Gamma_i]} &\le 
    \log(n+1) \cdot 2^{n+1} \cdot 
    \ifsv M^\alpha;\\ 
    \limsup_{i\to\infty} \frac{\rk_R H_k(M_i;R)}{[\Gamma:\Gamma_i]} &\le \ifsv M^\alpha.\\
  \end{align*}
\end{thm}

We note that by L\"uck's approximation theorem 
\[
	\limsup_{i\to\infty} \frac{\dim_\Q H_k(M_i;\Q)}{[\Gamma:\Gamma_i]}=\lim_{i\to\infty} \frac{\dim_\Q H_k(M_i;\Q)}{[\Gamma:\Gamma_i]}
\]
equals the $k$-th $L^2$-Betti number of $M$ provided $(\Gamma_i)_i$ is a residual 
chain, which means that $\Gamma_i<\Gamma$ is normal and the intersection of 
all~$\Gamma_i$ is the trivial group.

\subsection{Closed hyperbolic manifolds}
Let us now recall what is known about the three types of simplicial volume that we are dealing with for closed hyperbolic manifolds. 

A celebrated result by Gromov and Thurston states that the simplicial volume of a closed hyperbolic manifold is equal to the Riemannian volume divided by the volume $v_n$ of the regular ideal geodesic $n$-simplex in the hyperbolic space $\mathbb{H}^n$~\cite{vbc,thurston}. Notice that the regular ideal geodesic $n$-simplex is unique up to isometry. 

For closed hyperbolic surfaces, it is known that simplicial volume, stable integral simplicial volume, and integral foliated simplicial volume all coincide~\cite[Example~6.2]{loehpagliantini}.

For closed hyperbolic $3$-manifolds, the integral foliated simplicial
volume is equal to the simplicial volume~\cite[Theorem
  1.1]{loehpagliantini}, but the exact relation with stable integral
simplicial volume was unknown. In Section~\ref{subsec:wcapps} we
complete the picture in dimension $3$, showing that the three
considered flavours of simplicial volume are equal
(Corollary~\ref{cor:hyp3}).

\begin{thm}\label{3-dimensional:equality}
 Let $M$ be an oriented closed connected hyperbolic $3$-manifold. Then 
$$
\sv M=\ifsv M= \stisv M.
$$
\end{thm}

This result basically answers a question~\cite[Question~6.5]{FFM} in the affirmative that was originally stated for \emph{stable complexity} 
rather than stable integral simplicial volume (but these very close notions basically play the same role in all applications).
The proof is based on Theorem~\ref{intro-thm:wc}. 
Indeed, the fundamental group of a closed hyperbolic $3$-manifold has property $\EMD*$ (see Definition~\ref{def:EMD*}).
Hence, the equality of integral foliated
simplicial volume and stable integral simplicial volume for
$\EMD*$~groups yield the conclusion.

Theorem~\ref{3-dimensional:equality} admits the following geometric interpretation.
In their proof of the Ehrenpreis conjecture~\cite{KM}, 
Kahn and Markovic showed that every closed orientable hyperbolic surface $S$ 
has a finite covering that decomposes into pairs of pants whose boundary curves
have length arbitrarily close to an arbitrarily  big constant $R>0$.
Theorem~\ref{3-dimensional:equality} provides a sort of $3$-dimensional
version of this result. Namely, Theorem~\ref{3-dimensional:equality} is equivalent to the fact that, for every $\varepsilon>0$ and $R\gg 0$,  every closed hyperbolic
$3$-manifold $M$ has a finite covering $\widehat{M}$ admitting an integral fundamental cycle $z_{\widehat{M}}$ with the following property:
if $N$ is the number of singular simplices appearing in $z_{\widehat{M}}$, then at least $(1-\varepsilon)N$ simplices of $z_{\widehat{M}}$  
are  $\varepsilon$-close in shape to a regular simplex with edge length 
bigger than $R$.

For closed hyperbolic manifolds of dimension at least~$4$, the stable
integral simplicial volume is \emph{not} equal to the simplicial
volume. More precisely, Francaviglia, Frigerio, and Martelli proved
that the ratio between stable integral simplicial volume and 
simplicial volume is uniformly strictly bigger than~$1$~\cite[Theorem~2.1]{FFM}.

In Section~\ref{sec:hyp}, we generalize this result to integral foliated 
simplicial volume (Theorem~\ref{thm:hyp}), as vaguely suggested by 
Francaviglia, Frigerio, and Martelli~\cite[Question~6.4]{FFM}:

\begin{thm}\label{higher:hyp:thm}
  For all~$n \in \matN_{\geq 4}$ there is a constant~$C_n \in \R_{<1}$ with the following property: For all oriented closed connected hyperbolic $n$-manifolds~$M$ we have
  \[ \sv M \leq C_n \cdot \ifsv M. 
  \]
\end{thm}

For the proof, we notice that in dimension at least~$4$, the
dihedral angle of the regular ideal simplex does not divide $2\pi$ and
in the same spirit as for stable integral simplicial volume~\cite{FFM}
we show that foliated integral cycles cannot be used to produce
efficient fundamental cycles computing the simplicial volume.  Indeed,
every fundamental cycle contains simplices with volume significantly
smaller than~$v_n$ or there are overlappings producing loss of volume.
To prove our statement we need to carefully estimate this loss of volume.

\subsection{Closed amenable manifolds}
We will now refer to an oriented closed connected manifold with amenable fundamental group as \emph{closed amenable manifold}. It is well-known that the simplicial volume of closed amenable manifolds  vanishes~\cite{vbc,ivanov}. This result relies on bounded 
cohomology techniques that cannot be exploited in the context of integral coefficients.

For finite fundamental groups, integral foliated simplicial volume and stable integral simplicial are equal (and non-zero)~\cite[Corollary 6.3]{loehpagliantini}.
Moreover, if a manifold splits off an $S^1$-factor, then the integral foliated simplicial volume vanishes~\cite[Chapter~5.2]{mschmidt}.
Sauer introduced an invariant related to the integral foliated simplicial volume
and provided an upper bound of this invariant in terms of the minimal volume; moreover, for closed amenable aspherical manifolds this invariant vanishes~\cite[Section~3]{sauerminvol}. 

\begin{thm}\label{intro-thm:amenablevanishing}
  Let $M$ be an oriented closed connected aspherical manifold of
  non-zero dimension with amenable fundamental
  group~$\Gamma$. Let $\alpha=\Gamma\actson (X,\mu)$ be a free standard $\Gamma$-space. 
  Then
  \[  \ifsv M=\ifsv M^\alpha = 0. 
  \]
\end{thm}
    
The first statement of the next theorem is Corollary~\ref{cor:amenable}. The second statement about 
vanishing is a combination of Theorem~\ref{intro-thm:amenablevanishing} applied to the action of $\Gamma$ on its profinite completion (cf.~Subsection~\ref{subsec: actions}) and 
the fact that $\ifsv M ^{\pfc \Gamma} = \stisv M$ (Theorem~\ref{thm: relation foliated stable}). 

\begin{thm}
 Let $M$ be an oriented closed connected manifold with residually
  finite amenable fundamental group~$\Gamma$. Then 
  \[ \ifsv M = \ifsv M ^{\pfc \Gamma} = \stisv M. 
  \]
  where $\widehat \Gamma$ denotes the profinite completion of~$\Gamma$. If, in addition, $M$ is 
  aspherical, then 
  \[ \ifsv M = \ifsv M ^{\pfc \Gamma} = \stisv M=0. 
  \] 
\end{thm}

By Theorem~\ref{intro-thm:amenablevanishing} applied to the action of $\Gamma$ on the coset tree associated to a Farber chain  (cf.~Subsection~\ref{subsec: actions}) and by Theorem~\ref{intro-thm:homology bound} we obtain: 

\begin{thm}
 Let $M$ be an oriented closed connected aspherical manifold with amenable fundamental group~$\Gamma$. 
 Let $(\Gamma_i)_i$ be a Farber chain of $\Gamma$, and let $M_i\to M$ be the finite covering 
	associated to $\Gamma_i$. For every integer $k\ge 0$ and for every 
        principal ideal domain~$R$ we have
	\begin{align*}
		\limsup_{i\to\infty} \frac{\log |\tors H_k(M_i;\Z)|}{[\Gamma:\Gamma_i]} &=0;\\ 
		\limsup_{i\to\infty} \frac{\rk_R H_k(M_i;R)}{[\Gamma:\Gamma_i]} &=0.
	\end{align*}
\end{thm}

Using different methods, A.~Kar, P.~Kropholler and N.~Nikolov recently proved a more general form of the above theorem 
for simplicial complexes that are 
not necessarily aspherical but whose $k$-th homology of the universal covering vanishes~\cite{kar+nikolov}. 
Earlier, the above statement was shown for residual chains and under the assumption that $\Gamma$ has a normal 
infinite elementary amenable subgroup by L\"uck~\cite{lueck-homologygrowth} and 
for general amenable fundamental groups and residual chains by Sauer~\cite{sauer-homologygrowth}. 

Note that the middle equation in the above statement is a well known result; it follows from L\"uck's approximation theorem~\cite{lueck-approximation}   
and the vanishing of $L^2$-Betti numbers of amenable groups by Cheeger and Gromov~\cite{cheeger+gromov}. 

In view of the original, motivating problem about simplicial volume
and the Euler characteristic it would be interesting to know the
answer to the following question:

\begin{quest}
  Let $M$ be an oriented closed connected aspherical manifold with~$\sv M = 0$. 
  Does this imply~$\ifsv M = 0$? If $\pi_1(M)$ is residually finite, does 
  this imply~$\stisv M = 0$?
\end{quest}

\subsection{Analogies between integral foliated simplicial volume and cost}
We draw an analogy between integral foliated simplicial volume and cost. The latter invariant 
was introduced by Gaboriau~\cite{gaboriau-cost}. 

Let $\alpha=\Gamma\actson (X,\mu)$ be a standard $\Gamma$-action, and let $M$ be an $n$-dimensional 
closed aspherical manifold 
with $\pi_1(M)=\Gamma$. Then $\ifsv M^\alpha$ can be regarded as an invariant of $\alpha$. There is no 
direct relation between $\ifsv M^\alpha$ and the cost of~$\alpha$, nor between $\ifsv M$ and the 
cost of $\Gamma$ which is defined as the infimum of the costs of all free standard $\Gamma$-actions.  
However, there are similarities. The cost of $\Gamma$ can be thought of as an ergodic-theoretic 
version of the rank of $\Gamma$ (i.e., the minimal number of generators or, equivalently, the minimal number 
of $\Gamma$-orbits in a Caley graph of $\Gamma$). The integral foliated simplicial 
volume can be thought of as an ergodic-theoretic version of the minimal number of $\Gamma$-orbits of 
$n$-simplices in a simplicial model of the classifying space of $\Gamma$. Theorem~\ref{intro-thm:wc} 
was inspired by an analogous theorem for the cost by Abert and Weiss~\cite{abertweiss}. 
Gaboriau's fixed price problem asks whether the costs of two free standard $\Gamma$-actions always coincide. 
The following analog is also a more specific instance of Question~\ref{qu: dependency on prob space}. 

\begin{quest}[analog of fixed price problem]
	Let $M$ be an oriented closed connected aspherical manifold with fundamental group~$\Gamma$. 
	Let $\alpha$ and $\beta$ be free standard $\Gamma$-actions. Does $\ifsv M^\alpha=\ifsv M^\beta$ 
	hold?
\end{quest}

\subsection*{Organization of this article}
In Section~\ref{sec:ifsv}, we recall the exact definition of integral
foliated simplicial volume. The behaviour of integral foliated
simplicial volume with respect to weak containment and applications
thereof are studied in Section~\ref{sec:wc}. The homology bounds by integral 
foliated simplicial volume are discussed in Section~\ref{sec: homology bounds}. 
The higher-dimensional
hyperbolic case is treated in Section~\ref{sec:hyp}, the amenable
aspherical case in Section~\ref{sec:amenablevanishing}.

\subsection*{Acknowledgements}
C.L.~was supported by the CRC~1085 \emph{Higher Invariants}
(Universit\"at Regensburg, funded by the DFG) and is grateful to the
FIM at ETH Z\"urich for its hospitality. C.P.~was supported by Swiss National Science Foundation project 144373. 
The authors thank Alberto Abbondandolo and Pietro Majer for useful conversations.

\section{Integral foliated simplicial volume}\label{sec:ifsv}

Integral foliated simplicial volume mixes the rigidity of integral
coefficients with the flexibility of probability
spaces~\cite[p.~305f]{gromov}\cite{mschmidt}. In the following, we
recall the exact definition and collect some notation and terminology.
More background on integral foliated simplicial volume and its basic 
relations with simplicial volume and stable integral simplicial volume 
can be found in the literature~\cite{loehpagliantini}.

\subsection{Probability measure preserving actions}\label{subsec: actions}
A \emph{standard Borel space} is a measurable space that is isomorphic
to a Polish space with its Borel $\sigma$-algebra.  A \emph{standard
  Borel probability space} is a standard Borel space endowed with a
probability measure.  More information on the convenient category of
standard Borel spaces can be found in the book by
Kechris~\cite{kechrisdescriptive}.

Let $\Gamma$ be a countable group. A \emph{standard $\Gamma$-space}
is a standard Borel probability space~$(X,\mu)$ together with a
measurable $\mu$-preserving (left) $\Gamma$-action. If $\alpha =
\Gamma \actson (X,\mu)$ is a standard $\Gamma$-space, then we denote
the action of~$g \in \Gamma$ on~$x \in X$ also by~$g^\alpha(x)$.
Standard $\Gamma$-spaces~$\alpha = \Gamma \actson (X,\mu)$ and $\beta
= \Gamma \actson (Y,\nu)$ are \emph{isomorphic}, $\alpha \cong_\Gamma
\beta$, if there exist probability measure preserving
$\Gamma$-equivariant measurable maps~$X \longrightarrow Y$ and $Y
\longrightarrow X$ defined on subsets of full measure that are mutually inverse up to null sets.  A standard $\Gamma$-space is \emph{(essentially) free} or
\emph{ergodic} if the $\Gamma$-action is free on a subset of full measure 
or ergodic respectively. 
We describe two important examples of standard $\Gamma$-spaces: Bernoulli-shifts 
and profinite actions coming from chains of subgroups. 

Let $B$ be a 
standard
  Borel probability space. 
The \emph{Bernoulli shift of~$\Gamma$ with 
base~$B$} is the standard Borel space~$B^\Gamma$ with the product 
probability measure and the left translation action of~$\Gamma$.
If $\Gamma$ is an infinite countable group and $B$ is
non-trivial,
then the Bernoulli
shift~$B^\Gamma$ is essentially free and mixing (thus
ergodic)~\cite[p.~58]{petersen}.

A \emph{chain} in a group $\Gamma$ is a descending 
sequence $\Gamma=\Gamma_0>\Gamma_1>\Gamma_2>\ldots$ of finite index subgroups. 
The \emph{coset tree} $X$ of the chain is the inverse limit  
\[
	X=\varprojlim_{i \in \matN} \Gamma/\Gamma_i
\]
of the finite $\Gamma$-spaces $\Gamma/\Gamma_i$. Further, the profinite space~$X$ 
carries a $\Gamma$-invariant Borel probability measure~$\mu$ that is characterized by its pushforward to every 
$\Gamma/\Gamma_i$ being the normalized counting measure. The $\Gamma$-action on the standard $\Gamma$-space $(X,\mu)$ is 
ergodic~\cite[Chapter~3]{abert+nikolov}.   
If the chain consists of normal subgroups whose intersection is trivial (a  
so-called \emph{residual chain}), 
then the $\Gamma$-action on $X$ is essentially free. One calls the chain 
\emph{Farber} if the $\Gamma$-action on $(X,\mu)$ is essentially free; this 
notion also admits a group-theoretic characterization~\cite[(0-1) in Theorem~0.3.]{farber}. 

More generally, 
instead of taking the inverse limit over a chain of subgroups, one can also 
take an inverse limit over a system of subgroups, directed by inclusion. The 
\emph{profinite completion~$\pfc \Gamma$} 
  is defined as
  \[  \pfc \Gamma := \varprojlim_{\Lambda \in S} \Gamma/\Lambda
  \]
  where $S$ is the directed system of all finite index subgroups of~$\Gamma$. 
  Then $\pfc \Gamma$ is a profinite group. 
The unique Borel probability measure~$\mu$ that is pushed forward to the normalized counting measures 
on the finite quotients is the normalized Haar measure of $\pfc \Gamma$. Similarly 
to the case of coset trees, one sees that the 
left translation action of~$\Gamma$ on~$\pfc \Gamma$ is ergodic; this action is essentially free
if and only if $\Gamma$ is residually finite. 

\subsection{Parametrized fundamental cycles}

\begin{defn}[parametrized fundamental cycles]\label{def:parfc}
  Let $M$ be an oriented closed connected $n$-manifold with
  fundamental group~$\Gamma$ and universal covering~$\widetilde M
  \longrightarrow M$. 
  \begin{itemize}
    \item If $\alpha = \Gamma \actson (X,\mu)$ is a standard
      $\Gamma$-space, then we equip~$\linfz {X,\mu}$ with the right
      $\Gamma$-action
      \begin{align*}
        \linfz {X,\mu} \times \Gamma & \longrightarrow \linfz{X,\mu}\\
        (f,g) & \longmapsto g^\alpha (f) := \bigl(x \mapsto f(g^\alpha(x))\bigr).
      \end{align*}
      and we write $i_M^\alpha$ for the change of coefficients homomorphism
      \begin{align*}
        i_M^\alpha \colon C_*(M;\Z) \cong \Z \otimes_{\Z \Gamma} C_*({\widetilde M};\Z)
        & \longrightarrow \linfz {X,\mu} \otimes_{\Z \Gamma} 
        C_*({\widetilde M};\Z) =: C_*(M; \alpha)\\
        1 \otimes c 
        & \longmapsto 1 \otimes c
      \end{align*}
      induced by the inclusion~$\Z \hookrightarrow \linfz {X,\mu}$ as
      constant functions.
   \item If $\alpha = \Gamma \actson (X,\mu)$ is a standard $\Gamma$-space, then 
      \begin{align*} 
        \pfcl M  ^\alpha := H_n(i^\alpha_M)([M]_\Z) & \in
          H_n\bigl(M;\alpha) \\
        &= H_n\bigl(\linfz {X,\mu} \otimes_{\Z \Gamma} C_*({\widetilde M};\Z)
        \bigr)
        \end{align*}
        is the \emph{$\alpha$-parametrized fundamental class of~$M$}. All
        cycles in the chain complex~$C_*(M;\alpha) = \linfz {X,\mu} \otimes_{\Z\Gamma} C_*(\widetilde M;\Z)$ 
        representing~$\pfcl M^\alpha$ are called
        \emph{$\alpha$-parametrized fundamental cycles of~$M$}.
  \end{itemize}
\end{defn}

Integral foliated simplicial volume is the infimum of $\ell^1$-norms
over all parametrized fundamental cycles:

\begin{defn}[integral foliated simplicial volume]\label{def:ifsv}
  Let $M$ be an oriented closed connected $n$-manifold with
  fundamental group~$\Gamma$, and let $\alpha = \Gamma \actson
  (X,\mu)$ be a standard $\Gamma$-space.
  \begin{itemize}
   \item Let $c=\sum_{j=1}^k f_j\otimes \sigma_j \in C_*(M;\alpha)$ be a
     chain in \emph{reduced form}, i.e., the singular
     simplices~$\sigma_1, \dots, \sigma_k$ on~$\widetilde M$ satisfy
     $\pi\circ \sigma_j \neq \pi\circ \sigma_\ell$ for all~$j, \ell
     \in \{1,\dots,k\}$ with $j \neq \ell$ (where $\pi \colon
     \widetilde M \longrightarrow M$ is the universal covering
     map). Then we define
         \[    \ifsv c^\alpha 
            := \ifsv[bigg]{\sum_{j=1}^k f_j \otimes \sigma_j}^\alpha
            := \sum_{j=1}^k \int_X |f_j| \,d\mu \in \R_{\geqslant 0}.
         \]
     (Clearly, all reduced forms of a given chain lead to the same
         $\ell^1$-norm because the probability measure is
         $\Gamma$-invariant.)
   \item The
         \emph{$\alpha$-parametrized simplicial volume of~$M$},
         denoted by~$\ifsv M ^\alpha$, is the infimum of the $\ell^1$-norms
         of all $\alpha$-parametrized fundamental cycles of~$M$.  
   \item The
         \emph{integral foliated simplicial volume of~$M$}, denoted 
         by~$\ifsv M$, is the infimum of all~$\ifsv M ^\alpha$ over all 
         isomorphism classes of standard $\Gamma$-spaces~$\alpha$.  
  \end{itemize}
\end{defn}

\begin{rem}
  Notice that oriented closed connected manifolds have countable
  fundamental groups. If $M$ is an oriented closed connected manifold
  with fundamental group~$\Gamma$ and if $\alpha$ and $\beta$ are
  standard $\Gamma$-spaces with~$\alpha \cong_\Gamma \beta$, then
  $\ifsv M ^\alpha = \ifsv M ^\beta$.  Moreover, if $\Gamma$ is a
  countable group, then the class of isomorphism classes of standard
  $\Gamma$-spaces indeed forms a set~\cite[Remark~5.26]{mschmidt}.
\end{rem}

\begin{rem}\label{rem:red1norm}
  Let $M$ be an oriented closed connected manifold with fundamental
  group~$\Gamma$ and let $\alpha = \Gamma \actson (X,\mu)$ be a
  standard $\Gamma$-space. Let $D \subset \widetilde M$ be a
  (set-theoretical, strict) fundamental domain for the $\Gamma$-action
  on~$\widetilde M$ by deck transformations. Let $c = \sum_{j=1}^k f_j
  \otimes \sigma_j \in C_*(M;\alpha)$ be a chain where the (not
  necessarily distinct!) singular simplices~$\sigma_1, \dots,
  \sigma_k$ all have their $0$-vertex in~$D$, and where $f_1, \dots,
  f_k \in \linfz{X,\mu}$. Then
  \[ \ifsv c^\alpha
     = \ifsv[bigg]{\sum_{j=1}^k f_j \otimes \sigma_j}^\alpha
     = \ifsv[bigg]{\sum_{j=1}^k f_j \otimes \sigma_j}^{(X,\mu)},
  \]
  where $\ifsv{\cdot}^{(X,\mu)}$ is the corresponding $\ell^1$-norm
  on~$\linfz {X,\mu} \otimes_\Z C_*(\widetilde M;\Z)$. I.e., for
  chains that do not contain different singular simplices from the
  same $\Gamma$-orbit, we can compute the $\ell^1$-norm also in the
  non-equi\-variant chain complex. This will be convenient below when
  considering potentially different actions on the \emph{same} 
  probability space.  
\end{rem}

For the sake of completeness, we also describe the relation between
para\-metrized fundamental cycles and locally finite fundamental cycles
of the universal covering.

\begin{lemma}[parametrized fundamental cycles yield locally finite fundamental cycles]\label{lem:lffundcycle}
  Let $M$ be an oriented closed connected $n$-manifold with
  fundamental group~$\Gamma$, let $\widetilde M$ be its universal
  covering and let $\alpha = \Gamma \actson (X,\mu)$ be a standard
  $\Gamma$-space. Moreover, let $c = \sum_{j=1}^k f_j \otimes \sigma_j
  \in C_n(M;\alpha)$ be an $\alpha$-parametrized fundamental cycle
  of~$M$. Then for $\mu$-a.e.~$x \in X$ the chain
  \[ c_x := \sum_{\gamma \in \Gamma} \sum_{j=1}^k f_j(\gamma^{-1} \cdot x) \cdot \gamma \cdot \sigma_j  
  \]
  is a well-defined locally finite fundamental cycle of~$\widetilde M$.
\end{lemma}
\begin{proof}
  Let $B(\alpha,\Z)$ denote the abelian group of all (strictly)
  bounded, measurable, everywhere defined functions of type~$X
  \longrightarrow \Z$, and let $N(\alpha,\Z) \subset B(\alpha,\Z)$ be
  the subgroup of $\mu$-a.e.\ vanishing functions. Then $\linfz {X,\mu}
  = B(\alpha,\Z) / N(\alpha,\Z)$ as $\Z \Gamma$-modules, where we
  equip $B(\alpha,\Z)$ and $N(\alpha,\Z)$ with the obvious
  right $\Gamma$-actions.

  For~$x \in X$ there is a well-defined evaluation chain map
  \begin{align*}
    \varphi_x \colon B(\alpha,\Z) \otimes_{\Z \Gamma} C_*(\widetilde M;\Z) 
    & \longrightarrow C_*^{\mathrm{lf}}(\widetilde M;\Z)
    \\
    f \otimes \sigma & \longmapsto
    \sum_{\gamma \in \Gamma} f(\gamma^{-1} \cdot x) \cdot \gamma \cdot \sigma;
  \end{align*}
  notice that the sum on the right hand side indeed is locally finite 
  because $\Gamma$ acts properly discontinuously on~$\widetilde M$ by 
  deck transformations.

  Let $c_\Z \in C_n(M;\Z) \cong \Z \otimes_{\Z \Gamma} C_n(\widetilde
  M;\Z)$ be a fundamental cycle of~$M$. Then we can view $c_\Z$
  (via constant functions) as a chain in~$B(\alpha,\Z) \otimes_{\Z \Gamma}
  C_n(\widetilde M;\Z)$ and $\varphi_x(c_\Z)$ is the transfer of~$c_\Z$
  to~$\widetilde M$ and thus is a locally finite fundamental
  cycle of~$\widetilde M$. 
  
  If $c \in C_n(M;\alpha)$ is an $\alpha$-parametrized
  fundamental cycle, then $c$ can be represented by a chain~$c' \in
  B(\alpha,\Z) \otimes_{\Z \Gamma} C_n(\widetilde M;\Z)$ such that
  there exist $b \in B(\alpha,\Z) \otimes_{\Z \Gamma}
  C_{n+1}(\widetilde M;\Z)$ and $z \in N(\alpha,\Z) \otimes_{\Z \Gamma}
  C_n(\widetilde M;\Z)$ with~\cite[Remark~4.20]{loehpagliantini}
  \[ c' = c_\Z + \partial b + z. 
  \]
  Therefore, for $\mu$-a.e.~$x \in X$ we have 
  \begin{align*}
    c_x & = \varphi_x(c') 
          = \varphi_x(c_\Z + \partial b)
          = \varphi_z(c_\Z) + \partial \varphi_x (b),
  \end{align*}
  which is a locally finite fundamental cycle of~$\widetilde M$.
\end{proof}

\subsection{Relation between integral foliated simplicial volume and 
stable integral simplicial volume}

Actions on coset trees and the profinite completion provide the link between integral foliated simplicial volume and stable integral simplicial volume.

\begin{thm}\label{thm: relation foliated stable}
  Let $M$ be an oriented closed connected manifold with fundamental
  group~$\Gamma$. Then
    \[ \ifsv M \leqslant \ifsv M ^{\pfc \Gamma} = \stisv M. 
    \]
  More specifically, if $(\Gamma_i)_i$ is a chain of $\Gamma$ and $\alpha=\Gamma\actson (X,\mu)$ 
  the corresponding action on the coset tree, then 
  \[
  	\ifsv M^\alpha=\lim_{i\to\infty}\frac{\isv{M_i}}{[\Gamma:\Gamma_i]}. 
  \]
  where $M_i\to M$ is the covering associated to $\Gamma_i\subset\Gamma=\pi_1(M)$. 
\end{thm}

\begin{proof}
	The first statement is proved by L\"oh and Pagliantini~\cite[Theorem~6.6]{loehpagliantini}, and  
	the proof~\cite[Remark~6.7]{loehpagliantini} also shows that 
	\[\ifsv M^\alpha=\inf_{i\to\infty}\frac{\isv{M_i}}{[\Gamma:\Gamma_i]}.\]
	But the infimum is actually a limit: Let $(a(i))_i$ be a sequence in $\mathbb{N}$ converging to $\infty$ such 
	that 
	\[
		\liminf_{i\to\infty}\frac{\isv{M_i}}{[\Gamma:\Gamma_i]}=\lim_{i\to\infty}\frac{\isv{M_{a(i)}}}{[\Gamma:\Gamma_{a(i)}]}.
	\]
	The standard $\Gamma$-action on the coset tree associated to the chain $(\Gamma_{a(i)})_i$ is clearly isomorphic 
	to the one on the coset tree associated to the chain $(\Gamma_i)_i$. Hence by the aforementioned result of 
	L\"oh and Pagliantini we obtain that 
	\[
		\ifsv M^\alpha=\inf_{i\in\mathbb{N}} \frac{\isv{M_{a(i)}}}{[\Gamma:\Gamma_{a(i)}]}=\liminf_{i\to\infty}\frac{\isv{M_i}}{[\Gamma:\Gamma_i]}. 
	\]
	One argues analogously for the limit superior. This concludes the proof. 
\end{proof}

\section{Integral foliated simplicial volume and weak containment}\label{sec:wc}

Probability measure preserving actions are organized into a hierarchy by means of weak
containment. We recall the notion of weak containment and its main
properties in Section~\ref{subsec:wcprelim}.  In
Section~\ref{subsec:wcmonotonicity}, we will prove monotonicity of
integral foliated simplicial volume with respect to weak containment.
Some simple consequences of this monotonicity are discussed in
Section~\ref{subsec:wcapps}. 

\subsection{Weak containment}\label{subsec:wcprelim}

We first recall Kechris's notion of weak
containment~\cite{kechris,kechrisfree} and its relation with the weak
topology on the space of actions on a given standard Borel probability
space.

\begin{defn}[weak containment]
  Let $\Gamma$ be a countable group, and let $\alpha = \Gamma \actson
  (X,\mu)$ and $\beta = \Gamma \actson (Y,\nu)$ be standard
  $\Gamma$-spaces. Then \emph{$\alpha$ is weakly contained in~$\beta$}
  if the following holds: For all~$\varepsilon \in \R_{>0}$, all
  finite subsets~$F \subset \Gamma$, all~$m \in \matN$, and all Borel
  sets~$A_1, \dots, A_m \subset X$ there exist Borel subsets~$B_1,
  \dots, B_m \subset Y$ with
  \[ \forall_{\gamma \in F} \, \forall_{j,k \in \{1,\dots, m\}}\quad
     \bigl|
       \mu(\gamma^\alpha(A_j) \cap A_k)
       - \nu(\gamma^\beta(B_j) \cap B_k)
     \bigr|
     < \varepsilon.
  \]
  In this case, we write~$\alpha \prec \beta$. We call $\alpha$ and $\beta$ 
  \emph{weakly equivalent} if~$\alpha \prec \beta$ and $\beta \prec \alpha$. 
\end{defn}

For example, if the standard $\Gamma$-space $\alpha$ is a factor of a standard 
$\Gamma$-space~$\beta$, then $\alpha\prec\beta$ holds. 

We will use the following characterization of weak containment:

\begin{prop}[weak containment vs.\ weak closure~\protect{\cite[Proposition~10.1]{kechris}}]
  \label{prop:wcwc}
  Let $\Gamma$ be a countable group and let $\alpha = \Gamma \actson (X,\mu)$ and $\beta = \Gamma\actson (Y,\nu)$ be 
  non-atomic standard $\Gamma$-spaces. Then $\alpha \prec \beta$ if and 
  only if $\alpha$ lies in the closure of
  \[ \bigl\{ \gamma \in A(\Gamma, X, \mu) \bigm| \gamma \cong_\Gamma \beta \bigr\} 
  \]
  in~$A(\Gamma, X, \mu)$ with respect to the weak topology.
\end{prop}

Here, $A(\Gamma, X, \mu)$ denotes the set of all $\mu$-preserving
actions of~$\Gamma$ on the standard Borel probability space~$(X,\mu)$
by Borel isomorphisms. The weak topology on~$A(\Gamma,X,\mu)$ is
defined as follows: The set~$\Aut(X,\mu)$ of Borel automorphisms
of~$(X,\mu)$ carries a weak topology with respect to the family of all
evaluation maps~$\varphi \mapsto \varphi(A)$ associated with Borel
subsets~$A \subset X$. I.e., if $\varphi \in \Aut(X,\mu)$, then the 
family of sets of the type
\[ \bigl\{ \psi \in \Aut(X,\mu) 
   \bigm| \forall_{j \in \{1,\dots,m\}} \ \mu\bigl( \varphi(A_j) \symmdiff \psi(A_j) \bigr) < \delta
   \bigr\}
\]
where $\delta \in \R_{>0}$, $m \in \matN$, and $A_1, \dots, A_m \subset X$ 
are Borel subsets is an open neighbourhood basis of the weak topology 
on~$\Aut(X,\mu)$~\cite[Chapter~1(B)]{kechris}.
Viewing $A(\Gamma,X,\mu)$ as a subset of the 
product~$\Aut(X,\mu)^\Gamma$ then induces a topology on~$A(\Gamma,X,\mu)$, 
which is also called \emph{weak topology}. 

\subsection{Monotonicity of integral foliated simplicial volume under weak containment}\label{subsec:wcmonotonicity}

We now prove the following monotonicity result:

\begin{thm}[monotonicity of integral foliated simplicial volume]\label{thm:wc}
  Let $M$ be an oriented closed connected manifold with 
  fundamental group~$\Gamma$, and let $\alpha = \Gamma
  \actson (X,\mu)$ and $\beta = \Gamma\actson (Y,\nu)$ be free non-atomic standard
  $\Gamma$-spaces with~$\alpha \prec \beta$. Then
  \[ \ifsv M^\beta \leq \ifsv M^\alpha. 
  \]
\end{thm}

\begin{proof}
  Notice that in the case of finite fundamental group~$\Gamma$, every
  free standard $\Gamma$-space~$\alpha$ satisfies~$\ifsv M ^\alpha =
  1/|\Gamma| \cdot \isv {\widetilde M}$~\cite[Proposition~4.26 and
    Example~4.5]{loehpagliantini}. Therefore, we now focus on the
  infinite case.

  Let $n := \dim M$, let $c \in C_n(M;\alpha)$ be an
  $\alpha$-parametrized fundamental cycle of~$M$, and let $\varepsilon
  \in \R_{>0}$. Taking the infimum over all such fundamental cycles
  and all such~$\varepsilon$ shows that it is sufficient to prove
  $\ifsv M^\beta \leq \ifsv c ^\alpha + \varepsilon$. To this end, we
  show that there exists a standard $\Gamma$-space~$\gamma \in
  A(\Gamma,X,\mu)$ with $\gamma \cong_\Gamma \beta$ and $\ifsv M^\beta
  = \ifsv M ^\gamma \leq \ifsv c ^\alpha + \varepsilon$.

  As a first step, we write $c$ suitably in reduced form. Because $c$ is an
  $\alpha$-parametrized fundamental cycle of~$M$ there exist an
  integral fundamental cycle~$z \in C_n(M;\Z)$ and a chain~$b \in
  C_{n+1}(M;\alpha)$ with
  \[ c = z + \partial b \in C_n(M;\alpha);
  \]
  here, we view $C_*(M;\Z)$ as a subcomplex of~$C_*(M;\alpha)$ via the
  inclusion of~$\Z \hookrightarrow \linfz{X,\mu}$ as constant
  functions. Let $D \subset \widetilde M$ be a (set-theoretical,
  strict) fundamental domain for the deck transformation action
  of~$\Gamma$ on~$\widetilde M$. We can then write
  \begin{align*}
    z & = \sum_{\sigma \in S} a_\sigma \otimes \sigma \in C_{n}(M;\alpha), \\
    b & = \sum_{\tau \in T} f_\tau \otimes \tau \in C_{n+1}(M;\alpha),
  \end{align*}
  where $S \subset \map(\Delta^n, \widetilde M)$, $T \subset \map(\Delta^{n+1}, \widetilde M)$ 
  are finite subsets of singular simplices whose $0$-vertex lies in~$D$, and where 
  $f_\tau \in \linfz{X,\mu}$ are essentially bounded measurable functions and $a_\sigma \in \Z \subset \linfz{X,\mu}$ are constant 
  functions. Without loss of generality, we may assume that the $f_\tau$ are represented 
  as bounded (and not only essentially bounded) functions and that at least one 
  of the $f_\tau$ is not constant~$0$. 
  We then obtain in~$C_*(M;\alpha)$
  \begin{align*}
    c & = \sum_{\sigma \in S} a_\sigma \otimes \sigma 
        + \partial\biggl(\sum_{\tau \in T} f_\tau \otimes \tau\biggr) 
        \\
      & = \sum_{\sigma \in S} a_\sigma \otimes \sigma 
        + \sum_{j = 1}^{n+1} \sum_{\tau \in T} (-1)^j \cdot f_\tau \otimes (\tau \circ i_j)
        + \sum_{\tau \in T} g_\tau^\alpha(f_\tau) \otimes \tau_0;
  \end{align*}
  here, for $j \in \{0,\dots,n+1\}$, we write~$i_j \colon \Delta^n \longrightarrow \Delta^{n+1}$ 
  for the inclusion of the $j$-th face, and for $\tau \in T$, we let $g_\tau \in \Gamma$ be 
  the unique element satisfying 
  \[ g_\tau^{-1} \cdot (\tau \circ i_0)(e_0) 
     = g_\tau^{-1} \cdot \tau(e_1) \in D,
  \]
  and we set $\tau_0 := g_\tau^{-1} \cdot \tau$. Hence, in the above representation of~$c$ 
  all singular simplices have their $0$-vertex in~$D$. In view of Remark~\ref{rem:red1norm} 
  we therefore have
  \[ \ifsv c ^\alpha
     = \ifsv[bigg]{\sum_{\sigma \in S} a_\sigma \otimes \sigma 
        + \sum_{j = 1}^{n+1} \sum_{\tau \in T} (-1)^j \cdot f_\tau \otimes (\tau \circ i_j)
        + \sum_{\tau \in T} g_\tau^\alpha(f_\tau) \otimes \tau_0}^{(X,\mu)}.
  \]

  As next step we bring the characterization of weak containment via
  the weak topology (Proposition~\ref{prop:wcwc}) into play. 
  We choose a finite Borel partition~$X = A_1 \sqcup \dots \sqcup A_m$ of~$X$ 
  that is finer than the (finite) set 
  \[ \bigl\{ f_\tau^{-1}(k) \subset X \bigm| k \in \Z, \tau \in T \bigr\}, 
  \]
  and we consider
  \[ \delta := \frac{\varepsilon}{m \cdot \sum_{\tau \in T} \|f_\tau\|_\infty}
     \in \R_{>0} 
  \]
  as well as the finite set 
  \[ F := \{ g_\tau^{-1} \mid \tau \in T \} \subset \Gamma  
     .
  \]
  By Proposition~\ref{prop:wcwc} there is a standard $\Gamma$-space~$\gamma \in A(\Gamma,X,\mu)$ 
  with~$\gamma \cong_\Gamma \beta$ and 
  \[ \forall_{g \in F} \quad \forall_{j \in \{1,\dots, m\}} \quad
     \mu\bigl( g^\alpha (A_j) \symmdiff g^\gamma (A_j)  
        \bigr) 
     < \delta.
  \]

  Finally, we consider the chain~$c' \in C_n(M;\gamma)$ that is represented by the 
  chain
  \[ \sum_{\sigma \in S} a_\sigma \otimes \sigma 
        + \sum_{j = 1}^{n+1} \sum_{\tau \in T} (-1)^j \cdot f_\tau \otimes (\tau \circ i_j)
        + \sum_{\tau \in T} g_\tau^\gamma(f_\tau) \otimes \tau_0
  \]
  from~$\linfz{X,\mu} \otimes_\Z C_n(\widetilde M;\Z)$. Then the same calculation 
  as in the first step shows that
  \[ c' = z + \partial \biggl(\sum_{\tau \in T} f_\tau \otimes \tau \biggr)
  \]
  holds in~$C_*(M;\gamma)$ and that
  \[ \ifsv {c'}^\gamma 
     = \ifsv[bigg]
     {\sum_{\sigma \in S} a_\sigma \otimes \sigma 
        + \sum_{j = 1}^{n+1} \sum_{\tau \in T} (-1)^j \cdot f_\tau \otimes (\tau \circ i_j)
        + \sum_{\tau \in T} g_\tau^\gamma(f_\tau) \otimes \tau_0}^{(X,\mu)}.
  \]  
  In particular, $c'$ is a $\gamma$-parametrized fundamental cycle and 
  \begin{align*}
    & \bigl| \ifsv c ^\alpha  - \ifsv {c'}^\gamma
    \bigr|\\
    =\ & \biggl|
    \ifsv[bigg]{\sum_{\sigma \in S} a_\sigma \otimes \sigma 
        + \sum_{j = 1}^{n+1} \sum_{\tau \in T} (-1)^j \cdot f_\tau \otimes (\tau \circ i_j)
        + \sum_{\tau \in T} g_\tau^\alpha(f_\tau) \otimes \tau_0}^{(X,\mu)}\\
     -\
     & \ifsv[bigg]{\sum_{\sigma \in S} a_\sigma \otimes \sigma 
        + \sum_{j = 1}^{n+1} \sum_{\tau \in T} (-1)^j \cdot f_\tau \otimes (\tau \circ i_j)
        + \sum_{\tau \in T} g_\tau^\gamma(f_\tau) \otimes \tau_0}^{(X,\mu)}
      \biggr|\\
    \leq\ &
    \ifsv[bigg]
     {\sum_{\sigma \in S} a_\sigma \otimes \sigma 
        + \sum_{j = 1}^{n+1} \sum_{\tau \in T} (-1)^j \cdot f_\tau \otimes (\tau \circ i_j)
        + \sum_{\tau \in T} g_\tau^\alpha(f_\tau) \otimes \tau_0\\
        -\ & \biggl(
     \sum_{\sigma \in S} a_\sigma \otimes \sigma 
        + \sum_{j = 1}^{n+1} \sum_{\tau \in T} (-1)^j \cdot f_\tau \otimes (\tau \circ i_j)
        + \sum_{\tau \in T} g_\tau^\gamma(f_\tau) \otimes \tau_0\biggr)}^{(X,\mu)}\\
     \leq\ & 
     \sum_{\tau \in T} \bigl\| g_\tau^\alpha (f_\tau) - g_\tau^\gamma(f_\tau) \bigr\|_\infty.
  \end{align*}
  
  In the second step we used the (reverse) triangle inequality for~$\ifsv{\cdot}^{(X,\mu)}$; in the 
  third step we used the definition of the $\ell^1$-norm on~$\linfz{X,\mu} \otimes_\Z C_n(\widetilde M;\Z)$ 
  and the triangle inequality. 
  
  For each~$\tau \in T$, we can write
  \[ f_\tau = \sum_{j =1}^m a_{\tau,j} \cdot \chi_{A_j} \in \linfz{X,\mu} 
  \]
  with certain~$a_{\tau,1}, \dots, a_{\tau,m} \in \Z$. 
  Hence, 
  \begin{align*} 
    \bigl\| g_\tau^\alpha (f_\tau) - g_\tau^\gamma(f_\tau) \bigr\|_\infty 
    & \leq \sum_{j=1}^m |a_{\tau,j}| \cdot \bigl\| g_\tau^\alpha(\chi_{A_j}) - g_\tau^{\gamma}(\chi_{A_j}) \bigr\|_\infty
    \\
    & \leq \sum_{j=1}^m |a_{\tau,j}| \cdot \bigl\| \chi_{(g_\tau^{-1}){}^{\alpha}(A_j)} - \chi_{(g_\tau^{-1}){}^\gamma(A_j)} \bigr\|_\infty
    \\
    & = \sum_{j=1}^m |a_{\tau,j}| \cdot \mu\bigl( (g_\tau^{-1}){}^\alpha(A_j) \symmdiff (g_\tau^{-1}){}^\gamma(A_j)\bigr)
    \\
    & \leq m \cdot \|f_\tau\|_{\infty} \cdot \delta.
  \end{align*}
  Therefore, we have
  \[  \bigl| \ifsv c ^\alpha  - \ifsv {c'}^\gamma \bigr|
     \leq m \cdot \sum_{\tau \in T} \|f_\tau\|_\infty \cdot \delta
     \leq \varepsilon.
  \]
  In particular, $\ifsv M ^\beta = \ifsv M ^\gamma \leq \ifsv{c'}^\gamma \leq \ifsv c^\alpha + \varepsilon$, as 
  desired.
\end{proof}

\subsection{Consequences}\label{subsec:wcapps}

In the following, we will combine Theorem~\ref{thm:wc} with known results 
from ergodic theory on weak containment. In particular, we will consider 
Bernoulli shifts and the profinite completion of residually finite groups. 

\begin{cor}[maximality of Bernoulli shifts]\label{cor:maxbernoulli}
  Let $M$ be an oriented closed connected manifold with infinite
  fundamental group~$\Gamma$ and let $B$ be a non-trivial standard
  Borel probability space. Then 
  \[ \ifsv M^\alpha \leq \ifsv M^{B^\Gamma}
  \]
  holds for all free standard $\Gamma$-spaces~$\alpha$.
\end{cor}
\begin{proof}
  This follows directly from Theorem~\ref{thm:wc} and the fact 
  that any free standard $\Gamma$-space weakly contains all 
  non-trivial Bernoulli shifts~\cite{abertweiss}.
\end{proof}

Let us recall a measurable density notion by Kechris~\cite{kechrisfree}, related to the profinite completion:

\begin{defn}[Property $\EMD*$]\label{def:EMD*}
  An infinite countable group~$\Gamma$ has  property~$\EMD*$ if any
  ergodic standard $\Gamma$-space is weakly contained in the
  profinite completion~$\pfc \Gamma$ of~$\Gamma$. 
\end{defn}

\begin{cor}[profinite completion and stable integral simplicial volume]\label{cor:ifsvpfc}
  Let $M$ be an oriented closed connected manifold with fundamental
  group~$\Gamma$. If $\Gamma$ has~$\EMD*$, then for all ergodic standard $\Gamma$-spaces~$\alpha$ we have
    \[ \ifsv M ^\alpha \geqslant \ifsv M^{\pfc \Gamma} = \stisv M,  
    \]
    and hence $\ifsv M = \stisv M$. 
\end{cor}
\begin{proof}
   If $\Gamma$ has~$\EMD*$ and $\alpha$ is an ergodic standard $\Gamma$-space, 
  then we obtain from Theorem~\ref{thm:wc} that
  \[ \ifsv M ^\alpha \geqslant \ifsv M^{\pfc \Gamma} = \stisv M. 
  \]
  On the other hand, we know that integral foliated simplicial volume can be 
  computed by using ergodic standard $\Gamma$-spaces~\cite[Proposition~4.17]{loehpagliantini} 
  and that $\ifsv M \leq \stisv M$ (Theorem~\ref{thm: relation foliated stable}). Hence, 
  $\ifsv M = \stisv M$, as desired.
\end{proof}

\begin{cor}[amenable fundamental groups]\label{cor:amenable}
  Let $M$ be an oriented closed connected manifold with residually
  finite amenable fundamental group~$\Gamma$. Then
  \[ \ifsv M ^\alpha = \ifsv M ^{\pfc \Gamma} = \stisv M
  \]
  holds for all free standard $\Gamma$-spaces~$\alpha$. In particular,
  $\ifsv M = \stisv M$.
\end{cor}

\begin{proof}
  If $\Gamma$ is finite, then this is a straightforward
  calculation~\cite[Proposition~4.15,
    Corollary~4.27]{loehpagliantini}.

  If $\Gamma$ is infinite, then all free standard $\Gamma$-spaces are
  weakly equivalent~\cite{foremanweiss,kechris}; in particular, they
  are weakly equivalent to the profinite completion~$\pfc \Gamma$. Now
  the claim follows from Theorem~\ref{thm:wc} and
  Theorem~\ref{thm: relation foliated stable}.
\end{proof}

The case of aspherical manifolds with amenable fundamental
group will be discussed in Section~\ref{sec:amenablevanishing}.

\begin{cor}[free fundamental groups]\label{cor:free}
  Let $M$ be an oriented closed connected manifold with free fundamental 
  group~$\Gamma$. Then
  \[ \ifsv M = \ifsv M^{\pfc \Gamma} = \stisv M. 
  \]
\end{cor}
\begin{proof}
  If $\Gamma$ is trivial, then this chain of equalities clearly
  holds~\cite[Example~4.5]{loehpagliantini}. Moreover, it is known
  that free groups of non-zero rank
  satisfy~$\EMD*$~\cite{kechrisfree}. Hence, we can apply
  Corollary~\ref{cor:ifsvpfc}.
\end{proof}

\begin{rem}\label{rem:freenonvanishing}
  If $M$ is an oriented closed connected manifold with free
  fundamental group~$\Gamma$ of rank~$r$, then $\sv M =0$~\cite[p.~76]{loeh_phd}.  
  However, we will now see that if $r \geqslant 2$, then 
  \[ \ifsv M = \stisv M > 0. 
  \]
  Looking at the classifying map~$M \longrightarrow B \Gamma$ shows 
  that~\cite[Theorem~1.35(1), Example~1.36]{lueck}
  \[ b_1^{(2)}(M) = b_1^{(2)} (\widetilde M, \Gamma) 
     \geqslant b_1^{(2)}(\Gamma) = r -1 > 0. 
  \]
  Hence, the fact that integral foliated simplicial volume (and also
  stable integral simplicial volume) provide an upper bound for
  $L^2$-Betti numbers~\cite[Corollary~5.28]{mschmidt} implies that
  also~$\ifsv M = \stisv M > 0$.

  In case of rank~$r = 1$, then there are examples with vanishing
  stable integral simplicial volume (e.g., $S^1$), but also examples
  with non-vanishing stable integral simplicial volume (e.g., $(S^1
  \times S^{n-1}) \# (S^2 \times S^{n-2})$ for all~$n {\geqslant
    2}$, as one can see using the formula of $L^2$-Betti numbers for
  connected sums~\cite[Theorem~1.35(6)]{lueck} and the mentioned upper
  bound by Schmidt).
\end{rem}

Let us now consider the case of hyperbolic $3$-manifolds.
\begin{prop}\label{vf}
The fundamental group of a virtually fibered closed hyperbolic $3$-manifold has property~$\EMD*$.
\end{prop}
\begin{proof}
This result has already been noticed by Kechris~\cite{kechrisfree} and
Bowen and Tucker-Drob~\cite{BowenTucker}. For the sake of
completeness, we include a proof. 
First of all, notice that for residually finite groups property~$\EMD*$ is equivalent to 
property $\MD$~\cite[Theorem 1.4]{Tucker}, another universal property related to profinite completion 
due to Kechris~\cite{kechrisfree}.
Hence, we deduce that surface
groups have property~$\MD$~\cite[Theorem~1.4]{BowenTucker}: Indeed, using
as normal subgroup of a surface group the kernel of its abelianization
map, Lubotzky and Shalom~\cite[Theorem 2.8]{LubotskyShalom} showed
that a surface group satisfies the hypotheses for the $\MD$-inheritance result~\cite[Theorem 1.4]{BowenTucker}.

Let us now consider the fundamental group $\Gamma$ of a closed hyperbolic \mbox{$3$-mani}\-fold that fibers over $S^1$, i.e.,
the semidirect product of a surface group~$\Lambda$ and $\Z$. 
Taking as normal subgroup of $\Gamma$ the surface group $\Lambda$ we may again 
apply $\MD$-inheritance~\cite[Theorem~1.4]{BowenTucker} to   
conclude that $\Gamma$ has property~$\MD$, and hence property $\EMD*$.
For residually finite groups property~$\EMD*$ is preserved by passing from a finite index subgroup to the ambient group, hence yielding the conclusion.
\end{proof}

\begin{cor}[hyperbolic $3$-manifolds]\label{cor:hyp3}
 Let $M$ be an oriented closed connected hyperbolic $3$-mani\-fold. Then 
$$
\sv M=\ifsv M= \stisv M.
$$
\end{cor}
\begin{proof}
Agol's virtual fiber theorem \cite{Agol} and Proposition \ref{vf} give
that the fundamental group of \emph{every} closed hyperbolic
$3$-manifold has property~$\EMD*$. Hence, Corollary
\ref{cor:ifsvpfc} implies that integral foliated simplicial volume
and stable integral simplicial volume are equal for closed
hyperbolic $3$-manifolds.  On the other hand, it is known that $\sv M
= \ifsv M$ in this case~\cite[Theorem~1.1]{loehpagliantini}.
\end{proof}

Notice that the corresponding result for higher-dimensional manifolds does 
not hold (Section~\ref{sec:hyp}).

\begin{rem}
An example of group which does not satisfy property $\EMD*$ is $\text{SL}(n,\Z)$
for $n> 2$~\cite{BowenTucker}.
\end{rem}

\section{Homology bounds via integral foliated simplicial volume}\label{sec: homology bounds}

Theorem~\ref{intro-thm:homology bound}, which we prove in this section, can be quickly deduced 
from Theorem~\ref{thm: relation foliated stable} and suitable estimates for torsion and rank 
in terms of integral simplicial volume:

\begin{proof}[Proof of Theorem~\ref{intro-thm:homology bound}]
Every oriented closed connected $n$-manifold~$N$ satisfies the torsion homology bound~\cite[Theorem~3.2]{sauer-homologygrowth} 
\[
	\log(|\tors H_k(N;\Z)|)\le\log(n+1)\binom{n+1}{k+1}\isv N
\]
in every degree~$k$. 
In particular, by Theorem~\ref{thm: relation foliated stable} we obtain for the tower of finite coverings 
associated to the chain $(\Gamma_i)_i$  of $\Gamma=\pi_1(M)$ that 
\begin{align*}
	\limsup_{i\to\infty}\frac{\log(|\tors H_k(M_i;\Z)|)}{[\Gamma:\Gamma_i]}&\le\log(n+1)2^{n+1} \cdot \lim_{i\to\infty}
	\frac{\isv{M_i}}{[\Gamma:\Gamma_i]}\\
	&=\log(n+1)2^{n+1} \cdot \ifsv M^\alpha, 
\end{align*}
where $n$ denotes the dimension of~$M$. Starting from the Betti number estimate in 
Lemma~\ref{lem:bettiestimate} below we similarly obtain
\[ \limsup_{i \to\infty} \frac{\rk_R H_k(M_i;R)}{[\Gamma:\Gamma_i]} 
   \leq \ifsv M^\alpha
\]
for all principal ideal domains~$R$.
\end{proof}

Poincar\'e duality allows to bound Betti numbers in terms of integral simplicial volume. 
For the sake of completeness, we give a simple proof of this fact. For simplicity, we consider 
only principal ideal domains as coefficients.

\begin{lemma}\label{lem:bettiestimate}
  Let $R$ be a principal ideal domain and let $M$ be an oriented
  closed connected $n$-manifold. Then for all~$k \in \matN$ we have
  \[ \rk_R H_k(M;R) \leq \isv M. 
  \]
\end{lemma}
\begin{proof}
  Let $c = \sum_{j=1}^m a_j \cdot \sigma_j \in C_n(M;\Z)$ be a
  fundamental cycle of~$M$ in reduced form with~$|c|_{1,\Z} = \isv M$
  and let $k \in \matN$. Then the Poincar\'e duality map
  \begin{align*}
    \cdot \cap [M]_\Z \colon H^{n-k}(M;R) & \longrightarrow H_k(M;R) \\
    [f] & \longmapsto
    [f \cap c] = 
    (-1)^{k \cdot (n-k)}
    \cdot 
    \biggl[ \sum_{j=1}^m a_j \cdot f({}_{n-k} \lfloor\sigma) \cdot \sigma \rfloor_k
    \biggr]
  \end{align*}
  is surjective. In particular, $H_k(M;R)$ is a quotient of a submodule 
  of a free $R$-module of rank at most~$m$. So, 
  $\rk_R H_k(M;R)  \leq m \leq |c|_{1,\Z} = \isv M.
  $
\end{proof}

\section{Integral foliated simplicial volume of higher-dimensional hyperbolic manifolds}\label{sec:hyp}

This section is devoted to the proof of Theorem~\ref{higher:hyp:thm}:

\begin{thm}\label{thm:hyp}
  For all~$n \in \matN_{\geq 4}$ there is a~$C_n \in \R_{<1}$ with the following property: For all oriented closed connected hyperbolic $n$-manifolds~$M$ we have
  \[ \sv M \leq C_n \cdot \ifsv M. 
  \]
\end{thm}

\subsection{Setup}\label{subsec:hypsetup}

As usual, we denote by $\Gamma\cong \pi_1(M)$ the automorphism group of the universal covering $\pi\colon \widetilde{M}\to M$. We
also fix a standard $\Gamma$-space $X$, and we denote by $\alpha$ the action of $\Gamma$ on $X$.
Let 
$
c \in\ C_n(M,\alpha)=\linfz{X,\mu}\otimes_{\mathbb{Z}\Gamma} C_n(\widetilde{M};\mathbb{Z})
$
be a fundamental cycle for $M$. We  rewrite $c$ in a convenient form for the computations we are going to 
carry out. Namely, for a suitable finite set $I=\{1,\ldots,h\}$ of indices  we have
\begin{equation}\label{elementary}
c=\sum_{i\in I} \vare_i \chi_{A_i}\otimes \sigma_i\ ,
\end{equation}
where:
\begin{enumerate}
\item[(i)] the first vertex of each $\sigma_i$ lies in a
  fixed (set-theoretical, strict) fundamental domain $D$ for the $\Gamma$-action
  on~$\widetilde M$ by deck transformations;
 \item[(ii)] $\varepsilon_i\in \{1,-1\}$ for every $i\in I$;
 \item[(iii)] for every $i,j\in I$, either $A_i=A_j$, or $A_i\cap A_j=\emptyset$;
 \item[(iv)] if $A_i=A_j$ and $\sigma_i=\sigma_j$, then $\varepsilon_i=\varepsilon_j$.
\end{enumerate}
We set 
$$
\beta_i=\vare_i\mu(A_i)\in \R\ ,
$$
so
$$
\ifsv{c}^\alpha=\sum_{i\in I} |\beta_i|\ ,
$$
and we recall that
the chain
$$
c_\R=\sum_{i\in I} \beta_i (\pi\circ\sigma_i)= \sum_{i\in I} \vare_i\mu(A_i) (\pi\circ\sigma_i)\ \in C_n(M,\R) 
$$
is a real fundamental cycle for $M$~\cite[Remark~5.23]{mschmidt}\cite[Proposition~4.6]{loehpagliantini}. 

A crucial role in our argument will be played by the locally finite chain 
$$
c_x=\sum_{i\in I}\sum_{\gamma\in \Gamma} \vare_i \chi_{A_i}(\gamma^{-1}x)\cdot \gamma\sigma_i\in C_n^{\txt{lf}}(\widetilde M;\Z)
$$ 
defined for $\mu$-a.e.~$x\in X$, which is a 
locally finite fundamental cycle
of~$\widetilde M$ (Lemma~\ref{lem:lffundcycle}).

Observe that, as a consequence of our choices, if $\gamma\sigma_i=\gamma'\sigma_j$ for some $\gamma,\gamma'\in\Gamma$ and
$i,j\in I$, then $\gamma=\gamma'$ and $\sigma_i=\sigma_j$ (but possibly $i\neq j$). If this is the case and $\gamma^{-1}(x)\in A_i\cap A_j$, then
the singular simplex $\gamma\sigma_i=\gamma\sigma_j$ appears with multiplicities in $c_x$. However, we will 
keep track of the fact that these singular simplices arise from distinct summands in~\eqref{elementary}. In this spirit, the chain $c_x$ should be 
considered just as a sum of signed simplices (i.e., a locally finite chain whose coefficients lie in $\{\pm 1\}$), possibly with repetitions. 
To be precise
we say that the pair $(\gamma,i)\in\Gamma\times I$ \emph{appears in}
$c_x$ if $\gamma^{-1}(x)\in A_i$. Note that it could happen
that $\gamma\sigma_i=\gamma\sigma_j$ for some $i\neq j$, but $(\gamma,i)$ appears in $c_x$, while $(\gamma,j)$ does not.

By the very definitions, for every $i\in I,\gamma\in\Gamma$ we have
\begin{equation}\label{coefficient:equation}
\mu\left(\{x\in X\, |\, (\gamma,i)\ \textrm{appears\ in}\ c_x\}\right)=\mu(\gamma A_i)=\mu(A_i)=|\beta_i|\ .
\end{equation}

Moreover, we will assume that $c$ is straight as explained below. 

\subsection{Straight simplices}
Recall that a geodesic $p$-simplex in $\mathbb{H}^n$ is just the convex hull of $(p+1)$ (ordered) points
lying in $\overline{\mathbb{H}}^n=\mathbb{H}^n\cup \partial \mathbb{H}^n$. Such a simplex is \emph{ideal} if all its vertices belong to $\partial \mathbb{H}^n$, and it is
\emph{degenerate} if its vertices lie on a $(p-1)$-dimensional hyperbolic subspace of $\overline{\mathbb{H}}^n$. 
If $\sigma\colon \Delta^p\to \mathbb{H}^n$ is a singular simplex, then we denote by $\str_p(\sigma)$ the straight simplex associated
to $\sigma$, i.e., the barycentric parameterization
of the unique geodesic simplex having the same vertices as $\sigma$. It is well known that the straightening map can be linearly extended to an
${\rm Isom}(\mathbb{H}^n)$-equivariant chain map $\str_*\colon C_*(\mathbb{H}^n,\mathbb{Z})\to C_*(\mathbb{H}^n,\mathbb{Z})$,
which is ${\rm Isom}(\mathbb{H}^n)$-equivariantly homotopic to the identity~\cite[\S 11.6]{Ratcliffe}. As a consequence, the norm non-increasing map
$$
\id \otimes_{\Z \Gamma} \str_*\colon C_*(M,\alpha)\to C_*(M,\alpha)
$$
induces the identity in homology. Therefore, henceforth we will assume that each $\sigma_i$ in our parametrized foliated fundamental cycle $c$
is straight.

We define the \emph{algebraic volume} of $\sigma_i$ by setting
$$
\vola(\sigma_i)=\vare_i\int_{\sigma_i} \omega_{\widetilde{M}}\ ,
$$
where $\omega_{\widetilde{M}}$ is the volume form of $\widetilde{M}=\mathbb{H}^n$. Since $c_\R$ is a real fundamental cycle for $M$ we have the equality
\begin{equation}\label{fundamental:eq}
 \vol(M)=\sum_{i\in I} |\beta_i| \vola(\sigma_i)\ .
\end{equation}

The fact that parametrized integral cycles cannot be used to produce efficient cycles descends from the following observation:
in dimension greater than~$3$, the dihedral angle of the regular ideal simplex does not divide $2\pi$. 
As a consequence, for a.e.~$x\in X$ the integral (locally finite) cycle $c_x$ must contain a certain quantity of small simplices. This implies in turn that
$c_x$ cannot project onto an efficient fundamental cycle of $M$.
In the case when $c_x$ is $\Gamma$-equivariant, this fact  is precisely stated and proved by Francaviglia, Frigerio, and Martelli~\cite{FFM}, and implies that
the ratio between the (stable) integral simplicial volume and the ordinary simplicial volume
of $M$ is strictly bigger than one.
However, $c_x$ is not $\Gamma$-equivariant in general, so more work is needed in our context.

Let us first  recall the 
statements from the equivariant case~\cite{FFM} that we will be using later on. A \emph{ridge} of a geodesic simplex (or of a singular simplex) is a
face of the simplex of codimension~$2$.
If $\Delta$ is a nondegenerate geodesic $n$-simplex in $\mathbb{H}^n$ and
$E$ is a ridge $\Delta$, then the \emph{dihedral angle}
$\alpha (\Delta,E)$ 
of~$\Delta$ at $E$ is defined in the following way: let $p$ be a point in $E\cap \matH^n$, and let $H\subseteq \matH^n$ be the unique $2$-dimensional geodesic plane which intersects 
$E$ orthogonally  in~$p$. We define 
$\alpha(\Delta,E)$ as the angle in $p$ of the polygon
$\Delta\cap H$ of $H\cong \matH^2$. It is easily seen that this is well-defined (i.e., independent of $p$).
For every $n\geqslant 3$, 
we denote  the dihedral angle of the ideal regular $n$-dimensional simplex at any of its ridges by $\alpha_n$.

It is readily seen by intersecting the simplex with a horosphere
centered at any vertex that $\alpha_n$ equals the dihedral angle of
the regular  \emph{Euclidean} $(n-1)$-dimensional simplex at any of its
$(n-3)$-dimensional faces, so $\alpha_n=\arccos \frac 1{n-1}$. 
In particular, we have $\alpha_3=\arccos\frac{1}{2}=\frac \pi 3$. Since $\frac{2\pi}{6}<\arccos \frac{1}{3}<\frac{2\pi}{5}$ and
$\frac{2\pi}{5}<\arccos \frac{1}{n}<\frac{2\pi}{4}$ for every $n\geqslant 4$,
the real number $\frac{2\pi}{\alpha_n}$ is an integer if and only if $n=3$. 
For $n \in \matN_{\geq 4}$, we write~$k_n \in \matN$ for the unique
integer satisfying
$$k_n\alpha_n<2\pi<(k_n+1)\alpha_n;$$ 
then $k_n=5$ if $n=4$ and $k_n=4$ if $n\geqslant 5$.  If the volume of a geodesic simplex
is close to $v_n$, then the simplex must be close in shape to the regular ideal one, so one gets the following:

\begin{lemma}[{\cite[Lemma 3.16]{FFM}}]\label{trigonometry:lemma}
Let $n\geqslant 4$. Then, there exist $a_n>0$ 
and $\vare_n>0$, such that the 
following condition holds:
if a geodesic $n$-simplex $\Delta$ satisfies 
$\vol(\Delta)\geqslant (1-\vare_n)v_n$ and if 
$\alpha$ is the dihedral angle of $\Delta$
at any of its ridges, then
$$
\frac{2\pi}{k_n+1}(1+a_n) < \alpha < \frac{2\pi}{k_n}(1-a_n).
$$
\end{lemma}

\subsection{The incenter and inradius of a simplex}
Consider a nondegenerate geodesic $k$-simplex $\Delta\subseteq \matH^n$, and let $H(\Delta)\subseteq \mathbb{H}^n$ be the unique hyperbolic subspace
of dimension $k$ containing $\Delta$. We are going to recall the definition of \emph{inradius} $r(\Delta)$ of $\Delta$, which is 
due to Luo~\cite{Luo}.

For every point $p\in \Delta$ we denote by $r_\Delta(p)$ the radius of the maximal $k$-ball of $H(\Delta)$ centered in $p$ and contained
in $\Delta$, and we set
$$
r(\Delta):=\sup_{p\in \Delta} r_\Delta(p)\ \in \ (0,+\infty]\ .
$$
Since the volume of any $k$-simplex is smaller than $v_k$ and
the volume of $k$-balls diverges as the radius diverges, there exists a constant $r_k>0$ such that
$r_\Delta(p)\leqslant r_k$ for every  $p\in \Delta$, so $r(\Delta)\in (0,+\infty)$.   
Moreover, there is a unique point~$p\in \Delta$ with $r_\Delta(p)=r(\Delta)$ \cite{Luo}\cite[Lemma~3.15]{FFM}.
Such a point is denoted by the symbol $\inc(\Delta)$, and it is called the \emph{incenter} of $\Delta$.

The following lemma shows that,
in big simplices, the incenter of a face is uniformly distant from any other non-incident face.

\begin{lemma}[{\cite[Lemma 3.15]{FFM}}]\label{palle}
Let $n\geqslant 3$. There exist $\varepsilon_n>0$ and $\delta_n>0$ such
that the following holds for every geodesic $n$-simplex $\Delta\subseteq \mathbb{H}^n$ 
with $\vol(\Delta)\geqslant v_n(1-\varepsilon_n)$: let $e$ be any face of $\Delta$ and $e'$ another face of $\Delta$ that does not contain $e$; then
$$d(\inc(e), e') > 2\delta_n.$$
In particular, if $e,e'$ are distinct ridges of $\Delta$, then 
\[ B(\inc(e),\delta_n)\cap B(\inc(e'),\delta_n)=\emptyset.\]
\end{lemma}

Henceforth we fix constants $\vare_n>0$, $a_n>0$ and $\delta_n>0$ such that the statements of Lemmas~\ref{trigonometry:lemma} and~\ref{palle} hold.
We also set
$$
\eta_n=\vol(B(p,\delta_n))\ ,
$$
where $B(p,\delta_n)$ is any ball of radius $\delta_n$ in $\mathbb{H}^n$; notice 
that $\eta_n$ is independent of the point~$p$.

\subsection{Proof of Theorem~\ref{thm:hyp}}
Taking the infimum over all standard $\Gamma$-spaces~$X$ and over all
the fundamental cycles~$c\in
C_n(M,\alpha)$ shows that Theorem~\ref{higher:hyp:thm} will be a
consequence of the following:

\begin{thm}\label{quantitative:thm}
Let $n\in\mathbb{N}$, $n\geqslant 4$, and
$$
C_n:=\max \left\{1-\frac{\vare_n}{12}, 1-\frac{\eta_n}{3v_n}, 1-\frac{a_n\eta_n}{2v_n}\right\} < 1.$$
Then, for every oriented closed connected hyperbolic $n$-manifold $M$ with fundamental group $\Gamma$, every standard $\Gamma$-space 
$\alpha =
\Gamma \actson (X,\mu)$, and every fundamental cycle $c\in C_n(M,\alpha)$, the following inequality
holds:
$$
\|M\|\leqslant C_n  \cdot \ifsv{c}^\alpha .
$$
\end{thm}

The rest of this section is devoted to the proof of Theorem~\ref{quantitative:thm}. We will now keep the notation and assumptions introduced in Section~\ref{subsec:hypsetup}.

\begin{defn}
We say that the simplex $\sigma_i$ is \emph{big} if $\vola(\sigma_i)>(1-\vare_n)v_n$, and \emph{small} otherwise. We also set
$$
I_b=\{i\in I\, |\, \sigma_i\ \textrm{is\ big}\}\, ,\quad I_s=I\setminus I_b\ ,
$$
and
$$
c_b=\sum_{i\in I_b} \vare_i\chi_{A_i}\otimes \sigma_i\, ,\qquad
c_s=\sum_{i\in I_s} \vare_i\chi_{A_i}\otimes \sigma_i\ ,
$$
so that $c=c_b+c_s$. Also observe that we have
$\ifsv{c}^\alpha=\ifsv{c_b}^\alpha+\ifsv{c_s}^\alpha$. (Of course, $c_b$ and $c_s$ need not be cycles.)
\end{defn}

The following result says that a cycle is efficient only if its small simplices have a small weight.

\begin{prop}\label{obvious:estimate}
 Suppose that 
 $$
 \ifsv{c_s}^\alpha\geqslant \frac{\ifsv{c}^\alpha}{12}\ .
 $$
 Then 
 $$
 \ifsv{c}^\alpha\geqslant \frac{\|M\|}{1-\vare_n/12}\ .
 $$
\end{prop}
\begin{proof}
 By \eqref{fundamental:eq} we have
\begin{align*}
 \|M\|&=\frac{ \vol(M)}{v_n}=\frac{\sum_{i\in I} |\beta_i| \vola(\sigma_i)}{v_n}\\ &=
 \frac{\sum_{i\in I_b} |\beta_i| \vola(\sigma_i)}{v_n}
 +\frac{\sum_{i\in I_s} |\beta_i| \vola(\sigma_i)}{v_n}\\ &
 \leq \frac{\sum_{i\in I_b} |\beta_i| v_n}{v_n}+
 \frac{\sum_{i\in I_s} |\beta_i| (1-\vare_n)v_n}{v_n}=
 \sum_{i\in I} |\beta_i|-\vare_n\sum_{i\in I_s} |\beta_i|\\ &=\ifsv{c}^\alpha-\vare_n\sum_{i\in I_s} |\beta_i|
 =\ifsv{c}^\alpha-\vare_n\ifsv{c_s}^\alpha\leq \ifsv{c}^\alpha \left(1-\frac{\vare_n}{12}\right)
 \ .
\qedhere
 \end{align*}
 \end{proof}
 
Therefore, we are now left to consider the case when $\ifsv{c_s}^\alpha$ is small.
 
\subsubsection{Full ridges}
In the sequel we work with $\Delta$-complexes. These are variations of 
simplicial complexes in which distinct simplices may share more than one face 
and faces of the same simplex may be identified~\cite[\S 2.1]{hatcher}. 
To the real cycle $c_\R$ there is associated the finite $\Delta$-complex $P$ which is defined as follows. 
We take one copy $\Delta_i$ of the $n$-dimensional standard simplex $\Delta^n$, and we identify the $(n-1)$-dimensional
faces $F_1\subseteq \Delta_{i_1},\ldots,F_s\subseteq \Delta_{i_s}$ of these simplices 
if $\pi\circ\sigma_{i_1}\circ \partial_{j_1}=\dots=\pi\circ\sigma_{i_s}\circ\partial_{j_s}$, where
$\partial_{j_l}$ is the usual affine  identification between the $(n-1)$-dimensional standard simplex
and the face~$F_{l}$ of $\Delta_{i_l}$. Observe that, after identifying $\Delta_j$ with the standard simplex
$\Delta^n$ for every $j$, the singular simplices $\pi\circ\sigma_i$ glue into a continuous map $$f\colon P\to M\ .$$

 Let now $e$ be a ridge of $P$, and observe that the set of vertices of $e$ is endowed with a well-defined order.
 We denote by $\widetilde{e}$ the unique lift
 of $f(e)$ to $\mathbb{H}^n$ having the first vertex in $D$, so $\widetilde{e}$ is naturally a singular $(n-2)$-simplex.
 We shall often denote by $\widetilde{e}$ also the image of such a simplex, since this will not produce any confusion.
 If $S$ is a finite set, we denote by $|S|$ the cardinality of $S$.
 For every $x\in X$ we set
 \begin{align*}
 \Omega(e,x) & =\{(\gamma,i)\, |\, (\gamma,i)\ \textrm{appears\ in}\ c_x\ \textrm{and}\ \widetilde{e}\ \textrm{is\ a\ ridge\ of}
\ \gamma\sigma_i\}\subseteq \Gamma\times I\ , 
\\
 \Omega_b(e,x) & =\{(\gamma,i)\in \Omega(e,x)\, |\, \sigma_i\ \textrm{is\ big}\}\ ,
 \\
 N(e,x) & = |\Omega(e,x)|\ ,
 \\ 
 N_b(e,x) & =|\Omega_b(e,x)|
 \end{align*}
 (observe that, since $c_x$ is locally finite, the sets $\Omega_b(e,x)\subseteq \Omega(e,x)$ are finite).
 It is immediate to check that the functions $N(e,\cdot)\colon X\to \mathbb{N}$,
 $N_b(e,\cdot)\colon X\to \mathbb{N}$ are measurable.
We now consider the following measurable subsets of $X$:
 \begin{align*}
   F(e) & =\{x\in X\, |\, N_b(e,x)\geqslant k_n+1\}\ .
 \\
   \NF(e) & =\{x\in X\, |\, 1\leq N(e,x)\, ,\ N_b(e,x)\leq k_n\}\ ,
 \end{align*}
and we
say that $e$ is \emph{$x$-full} if $x\in F(e)$, and \emph{$x$-non-full} if $x\in \NF(e)$.  
 Observe that, thanks to Lemma~\ref{trigonometry:lemma},
if $e$ is $x$-full, then the big simplices appearing in $c_x$ must produce some overlapping around
$\widetilde{e}$. On the other hand, if $e$ is not $x$-full, then at least one small simplex 
appears around $\widetilde{e}$ in $c_x$.

Henceforth we denote by $E$ the set of $(n-2)$-dimensional (simplicial) faces of $P$, and for every $i\in I$ we
denote by $E(\sigma_i)$ the set of $(n-2)$-dimensional (singular) faces of $\sigma_i$, i.e., the set of ridges of $\sigma_i$.

Let us now fix~$i\in I$. We want to measure the fact that $\sigma_i$ may produce
overlappings in $c_x$
 for some $x\in X$. To this aim, we fix a ridge $\tau$ of $\sigma_i$. If $e_\tau\in E$ is the $(n-2)$-dimensional face
corresponding to $\tau$, then there exists a unique element $\gamma_\tau\in\Gamma$ such that $\gamma_\tau\cdot \tau=\widetilde{e}_\tau$
and we set
\begin{align*}
F(\sigma_i,\tau)  =F(\sigma_i,\tau,e_\tau)&=\{  x\in F(e_\tau)\,|\, (\gamma_\tau,i)\in \Omega(e_\tau,x)\}\ \subseteq X\ ,\\
\NF(\sigma_i,\tau) =\NF(\sigma_i,\tau,e_\tau)&=\{  x\in \NF(e_\tau)\,|\, (\gamma_\tau,i)\in \Omega(e_\tau,x)\}\ \subseteq X\ .
\end{align*}
Observe that, according to our definitions,
$$
F(\sigma_i,\tau)=\gamma_\tau A_i\cap F(e_\tau)\, ,\quad
\NF(\sigma_i,\tau)=\gamma_\tau A_i\cap \NF(e_\tau)\ .
$$
On the other hand, for later convenience  we set
\begin{align*}
 F(\sigma_i,\tau,e)= \emptyset =
\NF(\sigma_i,\tau,e)\ ,
\end{align*}
for every $e\in E\setminus \{e_\tau\}$. So
$$
F(\sigma_i,\tau)=\bigcup_{e\in E} F(\sigma_i,\tau,e)\, ,\qquad
\NF(\sigma_i,\tau)=\bigcup_{e\in E} \NF(\sigma_i,\tau,e)\ .
$$

Of course, for every $x\in X$ such that $(\gamma_\tau,i)$ appears in $c_x$ we have that $e_\tau$ is either $x$-full or $x$-non-full, so 
by the very definitions (and by Equation~\eqref{coefficient:equation})
 we have that
$$
|\beta_i|=\mu(F(\sigma_i,\tau))+\mu(\NF(\sigma_i,\tau))\ .
$$
By summing over the ridges of $\sigma_i$ we then get
\begin{equation}\label{FNF2:eq}
\frac{n(n+1)}{2} |\beta_i|=\sum_{\tau\in E(\sigma_i)}\bigl(\mu(F(\sigma_i,\tau))+\mu(\NF(\sigma_i,\tau))\bigr)\ . 
\end{equation}

\begin{lemma}\label{first:double:counting}
We have
$$
\frac{n(n+1)}{2}\ifsv{c_b}^\alpha\leq \sum_{i\in I_b} \sum_{\tau\in E(\sigma_i)} \mu\bigl(F(\sigma_i,\tau)\bigr)+\sum_{e\in E} 5\mu\bigl(\NF(e)\bigr)\ .
$$
\end{lemma}
\begin{proof}
From Equation~\eqref{FNF2:eq} we get
\begin{align*}
\frac{n(n+1)}{2}\ifsv{c_b}^\alpha & =
\frac{n(n+1)}{2}\sum_{i\in I_b} |\beta_i|\\ &=
\sum_{i\in I_b} \sum_{\tau\in E(\sigma_i)}\bigl(\mu(F(\sigma_i,\tau))+\mu(\NF(\sigma_i,\tau))\bigr)\ .
\end{align*}
Recall now that $k_n\leq 5$ for every $n\in\mathbb{N}$; so, by definition, around any \mbox{$x$-non-full} ridge 
$\widetilde{e}$ in $\mathbb{H}^n$ at most five big simplices of $c_x$ may appear. As a consequence for every $e\in E$ we have that
 $$
\sum_{i\in I_b} \sum_{\tau\in E(\sigma_i)}\mu(\NF(\sigma_i,\tau,e))\leq 
5\mu(\NF(e)) ,
$$
so
$$
\sum_{i\in I_b}\sum_{\tau\in E(\sigma_i)} \mu(\NF(\sigma_i,\tau))=\sum_{e\in E}
\sum_{i\in I_b}\sum_{\tau\in E(\sigma_i)} \mu(\NF(\sigma_i,\tau,e))\leq 5\sum_{e\in E}\mu(\NF(e))\ ,
$$
and this concludes the proof.
\end{proof}

The following proposition shows that, if the total weight  of big simplices of $c$ is large, then 
there must be many overlappings in $c_x$ for $x$ in a subset of large measure. 

\begin{prop}\label{eta:estimate}
Suppose that
$$
 \ifsv{c_s}^\alpha\leq \frac{\ifsv{c}^\alpha}{12}\ .
 $$
 Then 
 $$
\sum_{i\in I_b} \sum_{\tau\in E( \sigma_i)} \mu(F(\sigma_i,\tau))\geqslant 5\ifsv{c}^\alpha\ .
$$
\end{prop}
\begin{proof}

 Let $e\in E$ and take $x\in \NF(e)$. 
 By Lemma~\ref{trigonometry:lemma}, at least one small simplex must appear around
 $\widetilde{e}$ in $c_x$. In other words, we have that 
 $$
 \NF(e)\subseteq \bigcup_{i\in I_s} \bigcup_{\tau\in E( \sigma_i)} \NF(\sigma_i,\tau,e)\ ,
 $$
 so
 $$
 \mu(\NF(e))\leq \sum_{i\in I_s} \sum_{\tau\in E( \sigma_i)} \mu(\NF(\sigma_i,\tau,e))\ .
 $$
 By summing over $e\in E$ we obtain
 \begin{align*}
 \sum_{e\in E} \mu(\NF(e)) &\leq \sum_{e\in E} \sum_{i\in I_s} \sum_{\tau\in E(\sigma_i)} \mu(\NF(\sigma_i,\tau,e))
 =\sum_{i\in I_s} \sum_{\tau\in E(\sigma_i)} \mu(\NF(\sigma_i,\tau))
 \\ &\leq\frac{n(n+1)}{2} \sum_{i\in I_s}|\beta_i|=
 \frac{n(n+1)}{2} \ifsv{c_s}^\alpha\ ,
 \end{align*}
 where the second inequality follows from~\eqref{FNF2:eq}.
 Using this inequality and Lemma~\ref{first:double:counting}
we then get that
\begin{align*}
 \frac{n(n+1)}{2}\ifsv{c}^\alpha& =\frac{n(n+1)}{2}(\ifsv{c_b}^\alpha+\ifsv{c_s}^\alpha)\\ &
 \leq 
 \sum_{i\in I_b} \sum_{\tau\in E( \sigma_i)} \mu(F(\sigma_i,\tau))+\sum_{e\in E} 5\mu(\NF(e))
+\frac{n(n+1)}{2}\ifsv{c_s}^\alpha\\
 &\leq  \sum_{i\in I_b} \sum_{\tau\in E( \sigma_i)} \mu(F(\sigma_i,\tau))+5\frac{n(n+1)}{2}\ifsv{ c_s}^\alpha +\frac{n(n+1)}{2}\ifsv{ c_s}^\alpha\\ &=
 \sum_{i\in I_b} \sum_{\tau\in E( \sigma_i)} \mu(F(\sigma_i,\tau))+3{n(n+1)}\ifsv{c_s}^\alpha
 \\ &\leq  
  \sum_{i\in I_b} \sum_{\tau\in E( \sigma_i)} \mu(F(\sigma_i,\tau))
 +\frac{n(n+1)}{4}\ifsv{c}^\alpha\ ,
 \end{align*}
so
\[
\sum_{i\in I_b} \sum_{\tau\in E( \sigma_i)} \mu(F(\sigma_i,\tau))\geqslant \frac{n(n+1)}{4}\ifsv{c}^\alpha
\geqslant 5\ifsv{c}^\alpha\ .
\qedhere
\]
\end{proof}

\begin{rem}
Suppose that every simplex in $c$ has positive algebraic volume (this is very ``likely'' if $\ifsv{c}^\alpha$ is close to $\|M\|$). Then
for every $x\in X$ all the simplices in $c_x$ have positive algebraic volume. Since $c_x$
has integral coefficients, this readily implies that
no overlapping occurs in $c_x$. In particular, no full ridge appears in any $c_x$,
so $\mu(F(\sigma_i,\tau,e))=0$ for every $e\in E$, $i\in I$, $\tau\in E(\sigma_i)$.
By Proposition~\ref{eta:estimate}
and Proposition~\ref{obvious:estimate} we now have
$$
\ifsv{c_s}^\alpha\geqslant \frac{\ifsv{c}^\alpha}{12}\, \qquad \textrm{and} \qquad
\ifsv{c}^\alpha\geqslant \frac{\|M\|}{1-\vare_n/12}\ .
$$
This implies (a stronger version of) Theorem~\ref{quantitative:thm} in the case when the simplices appearing in $c$
 all have positive algebraic volume. 
\end{rem}

\subsection{Generalized chains and local degree}
Let $B(\mathbb{H}^n;\mathbb{R})$ be the real vector space having as a basis the set of measurable bounded subsets of $\mathbb{H}^n$, endowed with the obvious
action by $\Gamma$. 
A \emph{generalized chain} is an element of the space
$$
\mathbb{R}\otimes_{\mathbb{R}\Gamma} B(\mathbb{H}^n;\mathbb{R})\ .
$$
If $1\otimes B$ is an indecomposable element in $\mathbb{R}\otimes_{\mathbb{R}\Gamma} B(\mathbb{H}^n;\mathbb{R})$,
then
we define the \emph{local degree} $\deg_p(1\otimes B)\in\mathbb{N}$ of $1\otimes B$ at $p\in M$  
as the number of points in the intersection between $B$ and the preimage of $p$ via the universal covering map
$\pi\colon \mathbb{H}^n\to M$
(such a number is well-defined because $\Gamma$ permutes the fiber of $p$, and it is finite since $B$ is relatively compact and the fiber of $p$ is discrete). 
It is readily seen that the local degree extends to a well-defined linear map
$$
\deg_p\colon \mathbb{R}\otimes_{\mathbb{R}\Gamma} B(\mathbb{H}^n;\mathbb{R})\to \R\ .
$$
Let now $z=\sum_{\lambda\in\Lambda} a_\lambda\otimes B_\lambda$ be a fixed generalized chain. 
The number
$$
\vol(z)=\sum_{\lambda\in\Lambda} a_\lambda \vol(B_\lambda)
$$
is well-defined, and 
a  double-counting argument shows that
$$
\vol(z)=\int_M \deg_p(z)\, dp
$$
(observe that the map $\deg_{\cdot}(z)\colon M\to \R$ is measurable, since each $B_\lambda$ is Borel).
Let us set
$$
I_{\rm pos}=\{i\in I\, |\, \vola(\sigma_i)\geqslant 0\}\ ,
$$
so that $I_b\subseteq I_{\rm pos}$.
Since $c_\R$ is a real fundamental cycle and the boundary of the image of each $\sigma_i$ is a null set, we easily have that
$$\deg_p\biggl(\sum_{i\in I_{\rm pos}} |\beta_i|\im(\sigma_i)-\sum_{i\in I\setminus I_{\rm pos}} |\beta_i|\im(\sigma_i)\biggr) =1$$ 
for almost every $p\in M$.  

In particular, if we set
$$
\overline{z}_{\rm pos}=\sum_{i\in I_{\rm pos}} |\beta_i|\otimes \im(\sigma_i)\ ,
$$
then we deduce that
\begin{equation}\label{local:degree:pos}
 \deg_p (\overline{z}_{\rm pos}) \geqslant 1
\end{equation}
for almost every $p\in M$.
We now need to refine this estimate in order to show that overlappings in $c_x$ 
have a cost in terms of the $\ell^1$-norm of the fundamental cycle $c_\R$.

Take $e\in E$. If $F(e)\neq \emptyset$, then $\widetilde{e}\subseteq \mathbb{H}^n$ is a ridge of a 
big geodesic simplex. Since big simplices are nondegenerate, also $\widetilde{e}$ is nondegenerate. In particular, the incenter
$\inc(\widetilde{e})$ is well-defined. We denote by $\inc(e)\in M$ the image of~$\inc(\widetilde{e})$ in $M$.

Recall that, if $\sigma_i$ is a big simplex and $\tau\in E(\sigma_i)$, then 
there exists a unique element $\gamma_\tau\in \Gamma$ such that $\gamma_\tau\cdot \tau=\widetilde{e}_\tau$, where $\widetilde{e}_\tau$ is the preferred
lift of the $(n-2)$-dimensional face $e_\tau\in E$ corresponding to $\tau$.
We consider the following generalized chains:
$$
\overline{z}_-=\sum_{i\in I_b} \sum_{\tau\in E(\sigma_i)} \mu(F(\sigma_i,\tau)) \otimes \left((\gamma_\tau\cdot \im(\sigma_i))\cap B(\inc(\widetilde{e}_\tau),\delta_n)\right)
$$
(observe that $\gamma_\tau$ and $\inc(\widetilde{e}_\tau)$ are defined whenever $\mu(F(\sigma_i,\tau))\neq 0$),
$$
\overline{z}_+=\sum_{e\in E} \mu(F(e)) \otimes B(\inc(\widetilde{e}),\delta_n)
$$
(again, $\inc(\widetilde{e})$ is defined whenever $\mu(F(e))\neq 0$), and we finally set
\begin{equation}\label{eq: cycle}
\overline{z}=\overline{z}_{\rm pos}-\overline{z}_-+\overline{z}_+\ .
\end{equation}
Roughly speaking, the generalized chain $\overline{z}$ is obtained by removing from the big simplices in $\overline{z}_{\rm pos}$ the portions that cover small balls around
the incenters of full ridges, and adding back (a suitable weighted sum of) such small balls. We will show that the degree of $\overline{z}$ at almost every point of $M$
is still at least~$1$. However, when there are many overlappings, the absolute value of the volume of $\overline{z}_-$ is substantially bigger than
the volume of $\overline{z}_+$. These two facts imply that the volume of $\overline{z}_{\rm pos}$ must be considerably bigger than~$\vol(M)$. Therefore, the sum of the coefficients
appearing in $\overline{z}_{\rm pos}$ must be bigger than $\|M\|$, and this implies in turn the desired bound from below for $\ifsv{c}^\alpha$.

In order to pursue this strategy, we need to introduce the notion of \emph{generalized parametrized chain}.
Let $B(\mathbb{H}^n;\mathbb{Z})$ be the free abelian group having as a basis the set of measurable bounded subsets of $\mathbb{H}^n$, endowed with the obvious
action by $\Gamma$. 
By definition, 
a generalized parametrized chain is an element of the space
$$
L^\infty(X,\mu,\mathbb{Z})\otimes_{\mathbb{Z}\Gamma} B(\mathbb{H}^n;\mathbb{Z})\ .
$$
The integration  map $\int_X\colon L^\infty(X,\mu,\mathbb{Z})\to\R$ and the inclusion
$i\colon B(\mathbb{H}^n;\mathbb{Z})\to B(\mathbb{H}^n;\mathbb{R})$ induce a well-defined homomorphism
$$
\theta:= \int_X \otimes\, i \colon 
L^\infty(X,\mu,\mathbb{Z})\otimes_{\mathbb{Z}\Gamma} B(\mathbb{H}^n;\mathbb{Z})\to \R\otimes_{\mathbb{R}\Gamma} B(\mathbb{H}^n;\mathbb{R})\ .
$$
If $\widetilde{z}$ is a generalized parametrized chain and $p\in M$, then we set
$$
\deg_p (\widetilde{z})=\deg_p (\theta(\widetilde{z}))\ .
$$
Moreover, if $\widetilde{z}=\sum_{j\in J} f_j\otimes B_j$, then
for every $x\in X$ we can consider the locally finite formal sum
$$
\widetilde{z}_x=\sum_{j\in J} \sum_{\gamma\in\Gamma}  f_j(\gamma^{-1}x) \gamma(B_j)\ ,
$$
and we may define the local degree of such a sum at $\widetilde{p}\in \mathbb{H}^n$ by setting
$$
\deg_{\widetilde{p}}(\widetilde{z}_x)=\sum_{j\in J} \sum_{\gamma\in\Gamma}  f_j(\gamma^{-1}x) \chi_{\gamma(B_j)}(\widetilde{p})\ .
$$

\begin{lemma}\label{locdegree:lemma}
 Let $\widetilde{z}$ be a generalized parametrized chain. Then
$$
 \deg_{p} (\widetilde{z})=\int_{X} \deg_{\widetilde{p}} (\widetilde{z}_x)\, d\mu(x) \in  \mathbb{R}
$$
for every $p\in M$ and $\widetilde{p}\in\mathbb{H}^n$ such that $\pi(\widetilde{p})=p$.
\end{lemma}
\begin{proof}
By linearity, we may assume that $\widetilde{z}=\chi_A\otimes B$ for some measurable subsets $A\subseteq X$, $B\subseteq \mathbb{H}^n$. 
Moreover, if $k=|(B\cap \pi^{-1}(p))|$, then there exist elements $\gamma_1,\ldots,\gamma_k$ in $\Gamma$ such that 
$\chi_{\gamma(B)}(\widetilde{p})=\chi_B(\gamma^{-1}(\widetilde{p}))=1$ if and only if $\gamma=\gamma_h$ for some $h\in\{1,\ldots,k\}$. Therefore,
$$
\deg_{\widetilde{p}}(\widetilde{z}_x)=\sum_{\gamma\in\Gamma}  \chi_A(\gamma^{-1}x) \chi_{\gamma(B)}(\widetilde{p})=
\sum_{h=1}^k \chi_A(\gamma_h^{-1}x)=\sum_{h=1}^k \chi_{\gamma_h\cdot A}(x)\ ,
$$
and
\[
\int_X \deg_{\widetilde{p}}(\widetilde{z}_x)\, d\mu(x)= \sum_{h=1}^k \int_X \chi_{\gamma_h\cdot A}(x)\, d\mu(x)=k\mu(A)=\deg_p(\widetilde{z})\ .
\qedhere
\]
\end{proof}

Let us now return to the study of our parametrized cycle $c=\sum_{i\in I} \vare_i\chi_{A_i}\otimes \sigma_i$. From now on $\bar z$ denotes the 
generalized chain in~\eqref{eq: cycle}. 

\begin{prop}\label{difficult}
 
 For almost every $p\in M$, we have
 $
 \deg_p (\overline{z})\geqslant 1
 $.
\end{prop}
\begin{proof}

We set
\begin{align*}
\widetilde{z}_{\rm pos}&=\sum_{i\in I_{\rm pos}} \chi_{A_i}\otimes \im(\sigma_i)\ ,\\
\widetilde{z}_-&=\sum_{i\in I_b} \sum_{\tau\in E(\sigma_i)} \chi_{F(\sigma_i,\tau)}\otimes\left((\gamma_\tau\cdot \im(\sigma_i))\cap B(\inc(\widetilde{e}_\tau),\delta_n)\right)\ ,\\
\widetilde{z}_+&=\sum_{e\in E} \chi_{F(e)}\otimes B(\inc(\widetilde{e}),\delta_n)\ ,
\end{align*}
and
$$
\widetilde{z}=\widetilde{z}_{\rm pos}-\widetilde{z}_-+\widetilde{z}_+\ .
$$
It follows from our choices that 
$\theta(\widetilde{z})=\overline{z}$, so 
$\deg_p (\widetilde{z})=\deg_p(\overline{z})$ for every point~$p\in M$. Therefore, by Lemma~\ref{locdegree:lemma} we are left to show that
$$
\deg_{\widetilde{p}} (\widetilde{z}_x)\geqslant 1
$$
for almost every $\widetilde{p}\in \mathbb{H}^n$ and almost every $x\in X$.

Let $i\in I_b$. Recall that 
$$
F(\sigma_i,\tau)=\gamma_\tau A_i\cap F(e_\tau)\ .
$$ 
Therefore, for every $\gamma\in\Gamma$ and every $x\in X$ we have 
$$
\chi_{F(\sigma_i,\tau)}(\gamma^{-1}(x))=\chi_{\gamma F(\sigma_i,\tau)}(x)=
\chi_{\gamma\gamma_\tau A_i\cap \gamma F(e_\tau)}(x)\ .
$$
Then, for every $x\in X$ and $\widetilde{p}\in \mathbb{H}^n$ we have
\begin{equation}\label{bigsum}
\begin{array}{lll}
 & \displaystyle\sum_{\gamma\in\Gamma}  \sum_{\tau\in E(\sigma_i)} \chi_{F(\sigma_i,\tau)}(\gamma^{-1}(x))\chi_{\gamma\gamma_\tau(\im(\sigma_i))\cap 
 \gamma B(\inc(\widetilde{e}_\tau),\delta_n)} 
 (\widetilde{p})\\ 
 = &\displaystyle
  \sum_{\gamma\in\Gamma}  \sum_{\tau\in E(\sigma_i)} \chi_{\gamma\gamma_\tau A_i\cap \gamma F(e_\tau)}(x)
  \chi_{\gamma\gamma_\tau(\im(\sigma_i))\cap 
 \gamma B(\inc(\widetilde{e}_\tau),\delta_n)} 
 (\widetilde{p})\\ 
  = &\displaystyle
 \sum_{\gamma\in\Gamma} \sum_{\tau\in E(\sigma_i)} \chi_{\gamma A_i\cap \gamma\gamma_\tau^{-1} F(e_\tau)}(x)
  \chi_{\gamma(\im(\sigma_i))\cap 
 \gamma\gamma_\tau^{-1} B(\inc(\widetilde{e}_\tau),\delta_n)} 
 (\widetilde{p})\\
  = &\displaystyle
 \sum_{\gamma\in\Gamma}  \sum_{\tau\in E(\sigma_i)} \chi_{\gamma A_i\cap \gamma\gamma_\tau^{-1} F(e_\tau)}(x)
  \chi_{\gamma(\im(\sigma_i))\cap 
 \gamma B(\inc(\gamma_\tau^{-1}\widetilde{e}_\tau),\delta_n)} 
 (\widetilde{p})\ .
 \end{array}
\end{equation}
Recall now that, since $\sigma_i$ is big, the $\delta_n$-balls centered at the incenters of the ridges of $\sigma_i$ are disjoint. As a consequence
we have that 
\begin{align*}
\chi_{\gamma A_i}(x) \chi_{\gamma\im(\sigma_i)}(\widetilde{p}) & 
\geqslant \sum_{\tau\in E(\sigma_i)} \chi_{\gamma A_i}(x)\chi_{\gamma(\im(\sigma_i))\cap 
 \gamma B(\inc(\gamma_\tau^{-1}\widetilde{e}_\tau),\delta_n)}(\widetilde{p})
\\ & \geqslant \sum_{\tau\in E(\sigma_i)} \chi_{\gamma A_i\cap \gamma\gamma_\tau^{-1} F(e_\tau)}(x)\chi_{\gamma(\im(\sigma_i))\cap 
 \gamma B(\inc(\gamma_\tau^{-1}\widetilde{e}_\tau),\delta_n)}(\widetilde{p})\ . 
\end{align*}
By summing over $\gamma\in\Gamma$ and using~\eqref{bigsum} we obtain 
\begin{align*}
& \sum_{\gamma\in\Gamma} \chi_{\gamma A_i}(x) \chi_{\gamma\im(\sigma_i)}(\widetilde{p})\\ \geqslant &
\sum_{\gamma\in\Gamma}  \sum_{\tau\in E(\sigma_i)} \chi_{F(\sigma_i,\tau)}(\gamma^{-1}(x))\chi_{\gamma\gamma_\tau(\im(\sigma_i))\cap 
 \gamma B(\inc(\widetilde{e}_\tau),\delta_n)} 
 (\widetilde{p})\ .
\end{align*}
Finally, by summing over $i\in I_{\rm pos}$ we get
\begin{align*}
& \sum_{i\in I_{\rm pos}} \sum_{\gamma\in\Gamma} \chi_{\gamma A_i}(x) \chi_{\gamma\im(\sigma_i)}(\widetilde{p})\\ \geqslant 
& \sum_{i\in I_b} \sum_{\gamma\in\Gamma} \chi_{\gamma A_i}(x) \chi_{\gamma\im(\sigma_i)}(\widetilde{p})\\ \geqslant &
\sum_{i\in I_b} \sum_{\gamma\in\Gamma}  \sum_{\tau\in E(\sigma_i)} \chi_{F(\sigma_i,\tau)}(\gamma^{-1}(x))\chi_{\gamma\gamma_\tau(\im(\sigma_i))\cap 
 \gamma B(\inc(\widetilde{e}_\tau),\delta_n)} 
 (\widetilde{p})\ ,
\end{align*}
which just means that
$$
\deg_{\widetilde{p}} (\widetilde{z}_{{\rm pos},x})\geqslant \deg_{\widetilde{p}}(\widetilde{z}_{-,x})\ .
$$
Therefore, $\deg_{\widetilde{p}}(\widetilde{z}_x)\geqslant 0$ for every $x\in X$, $\widetilde{p}\in\mathbb{H}^n$. In order to conclude it is sufficient
to show that, if $x$ is such that $c_x$ is a fundamental cycle for $\mathbb{H}^n$, then $\deg_{\widetilde{p}}(\widetilde{z}_x)\geqslant 1$ for almost every $\widetilde{p}\in\mathbb{H}^n$. 
Let us fix such an $x\in X$. Since $c_x$ is a fundamental cycle and the boundary of the image of each (translate of) $\sigma_i$ is a null set, 
for almost every $\widetilde{p}\in\mathbb{H}^n$
we have
$$
\sum_{i\in I_{\rm pos}}\sum_{\gamma\in \Gamma} \chi_{A_i}(\gamma^{-1}x)\chi_{\im(\gamma\sigma_i)}(\widetilde{p})-
\sum_{i\in I\setminus I_{\rm pos}}\sum_{\gamma\in \Gamma} \chi_{A_i}(\gamma^{-1}x)\chi_{\im(\gamma\sigma_i)}(\widetilde{p})=1\ ,
$$
so
$$
\deg_{\widetilde{p}}(\widetilde{z}_{{\rm pos},x})=\sum_{i\in I_{\rm pos}}\sum_{\gamma\in \Gamma} \chi_{A_i}(\gamma^{-1}x)\chi_{\im(\gamma\sigma_i)}(\widetilde{p})\geqslant 1\ .
$$
 Therefore,
if $\deg_{\widetilde{p}}(\widetilde{z}_{-,x})=0$, then 
$\deg_{\widetilde{p}}(\widetilde{z}_x)\geqslant \deg_{\widetilde{p}}(\widetilde{z}_{{\rm pos},x})\geqslant 1$, and we are done.
Otherwise, 
we have
$\deg_{\widetilde{p}}(\widetilde{z}_{-,x})\geqslant 1$. Therefore, there exist $i\in I_{\rm pos}$, $\tau\in E(\sigma_i)$ and $\gamma\in\Gamma$ such that
$$
\chi_{F(\sigma_i,\tau)}(\gamma^{-1}(x))\chi_{\gamma\gamma_\tau(\im(\sigma_i))\cap 
 \gamma B(\inc(\widetilde{e}_\tau),\delta_n)} 
 (\widetilde{p})=1\ ,
 $$
 i.e.,
 $$
 x\in \gamma F(\sigma_i,\tau,e_\tau)\subseteq \gamma F(e_\tau) \qquad \textrm{and}\qquad 
 \widetilde{p}\in \gamma B(\inc(\widetilde{e}_\tau),\delta_n)\ .
 $$
This implies at once that
$$
\deg_{\widetilde{p}}(\widetilde{z}_{+,x})\geqslant 1\ ,
$$
so 
$$
\deg_{\widetilde{p}}(\widetilde{z}_x)=\left(\deg_{\widetilde{p}}(\widetilde{z}_{{\rm pos},x})- \deg_{\widetilde{p}}(\widetilde{z}_{-,x})\right)
+ \deg_{\widetilde{p}}(\widetilde{z}_{+,x})\geqslant 0+1\geqslant 1\ ,
$$
whence the conclusion.
\end{proof}

\begin{prop}\label{final:estimate:prop}
We have
\begin{equation}\label{final:estimate}
\ifsv{c}^\alpha v_n-\frac{1+a_n}{k_n+1}\eta_n\sum_{i\in I_b} \sum_{\tau\in E(\sigma_i)} \mu(F(\sigma_i,\tau))
+\eta_n \sum_{e\in E} \mu(F(e))\geqslant \vol(M)\ .
\end{equation}
\end{prop}
\begin{proof}
Of course we have
$$
\vol(\overline{z}_{\rm pos})=\sum_{i\in I_{\rm pos}} |\beta_i|\vol(\im(\sigma_i))\leq v_n\sum_{i\in I_{\rm pos}} |\beta_i|\leq \ifsv{c}^\alpha v_n
$$
and
$$
\vol(\overline{z}_+)=\sum_{e\in E} \mu(F(e)) \vol(B(\inc(\widetilde{e},\delta_n)))= \eta_n \sum_{e\in E} \mu(F(e))\ .
$$
Now recall that $\delta_n$ was chosen in such a way that, if $i\in I_b$, then the boundary of  
the ball of radius $\delta_n$ centered at the incenter of any ridge~$\tau$
of $\im(\sigma_i)$ does not meet any $(n-1)$-face of $\im(\sigma_i)$ not containing $\tau$. As a consequence, 
if $\mu(F(\sigma_i,\tau))\neq 0$, then
$\vol((\gamma_\tau\cdot \im(\sigma_i))\cap B(\inc(\widetilde{e}_\tau),\delta_n))\geqslant \eta_n(1+a_n)/(k_n+1)$. 
Therefore, we have
\begin{align*}
  \vol(\overline{z}_-) &= 
  \sum_{i\in I_b} \sum_{\tau\in E(\sigma_i)} \mu(F(\sigma_i,\tau)) \vol\left((\gamma_\tau\cdot \im(\sigma_i))\cap B(\inc(\widetilde{e}_\tau),\delta_n)\right)\\ &\geqslant
  \frac{1+a_n}{k_n+1}\eta_n  \sum_{i\in I_b} \sum_{\tau\in E(\sigma_i)} \mu(F(\sigma_i,\tau))\ .
  \end{align*}
  Therefore, we get that
  $$
\vol(\overline{z})\leq   \ifsv{c}^\alpha v_n-\frac{1+a_n}{k_n+1}\eta_n\sum_{i\in I_b} \sum_{\tau\in E(\sigma_i)} \mu(F(\sigma_i,\tau))
+\eta_n \sum_{e\in E} \mu(F(e))\ .
$$
On the other hand, 
by Proposition~\ref{difficult} we have
$$
\vol(M)\leq \int_M \deg_p(\overline{z})\, dp=\vol(\overline{z})\ ,
$$
and this concludes the proof.
\end{proof}

\begin{prop}\label{stima2}
 Suppose that $\ifsv{c_s}^\alpha\leq \ifsv{c}^\alpha/12$. Then
 $$
 \|M\| \leq \max\left\{ 1-\frac{\eta_n}{3v_n}, 1-\frac{a_n \eta_n}{2 v_n}\right\} \cdot \ifsv{c}^\alpha\ .
$$
\end{prop}
\begin{proof}
We divide the proof in two cases.
Suppose first that
$$\sum_{e\in E} \mu(F(e))\leq \frac{\ifsv{c}^\alpha}2\ .$$
Since $k_n+1\leqslant 6$ and $a_n>0$,
from Lemma~\ref{eta:estimate} we deduce that
$$
\frac{1+a_n}{k_n+1}\sum_{i\in I_b} \sum_{\tau\in E(\sigma_i)} \mu(F(\sigma_i,\tau))\geqslant \frac{5}{6} \ifsv{c}^\alpha\ .
$$
Plugging these inequalities into~\eqref{final:estimate}
we obtain
$$
\|M\|=\frac{\vol(M)}{v_n}\leqslant \ifsv{c}^\alpha+\frac{\eta_n}{v_n}\left(\frac{\ifsv{c}^\alpha}{2}-\frac{5\ifsv{c}^\alpha}{6}\right)
=\left(1-\frac{\eta_n}{3 v_n}\right) \ifsv{c}^\alpha .
$$

Now assume that
$\sum_{e\in E} \mu(F(e))\geqslant {\ifsv{c}^\alpha}/{2}$.
Recall that  around any $x$-full ridge
 $\widetilde{e}$ in $\mathbb{H}^n$ at least $k_n+1$ big simplices of $c_x$ appear. As a consequence for every $e\in E$ we have that
 $$
\sum_{i\in I_b} \sum_{\tau\in E(\sigma_i)}\mu(F(\sigma_i,\tau,e))\geqslant 
(k_n+1)\mu(F(e)) ,
$$
so
 $$
 \frac{1}{k_n+1}
\sum_{e\in E}\sum_{i\in I_b} \sum_{\tau\in E(\sigma_i)}\mu(F(\sigma_i,\tau,e))\geqslant 
\sum_{e\in E}\mu(F(e)) .
$$
Substituting this inequality into~\eqref{final:estimate} we get
\begin{align*}
\|M\|=\frac{\vol(M)}{v_n}&\leqslant \ifsv{c}^\alpha+\frac{\eta_n}{v_n}\left(\sum_{e\in E} \mu(F(e))-(1+a_n)\sum_{e\in E} \mu(F(e))\right)\\&=\ifsv{c}^\alpha-\frac{a_n\eta_n}{v_n}\sum_{e\in E} \mu(F(e))\leqslant 
\ifsv{c}^\alpha-\frac{a_n\eta_n\ifsv{c}^\alpha}{2v_n}.
\qedhere
\end{align*}
\end{proof}

The results proved in Propositions~\ref{obvious:estimate} and \ref{stima2} readily imply 
Theorem~\ref{quantitative:thm}.

\section{Integral foliated simplicial volume of aspherical manifolds with amenable fundamental group}\label{sec:amenablevanishing}

The section is devoted to the proof of Theorem~\ref{intro-thm:amenablevanishing}. 
Before we start, we remark that we cannot generalize Theorem~\ref{intro-thm:amenablevanishing} 
to rationally essential manifolds. 

\begin{rem}
  The statement of Theorem~\ref{intro-thm:amenablevanishing} does not hold
  if we replace ``aspherical'' by ``rationally essential''. For
  example, the oriented closed connected manifold~$M := (S^1)^4 \#
  (S^2 \times S^2)$ is rationally essential and has amenable
  residually finite fundamental group, but an $L^2$-Betti number
  estimate as in Remark~\ref{rem:freenonvanishing} shows that $\stisv
  M = \ifsv M >0$. 
\end{rem}

\subsection{Strategy of proof of Theorem~\ref{intro-thm:amenablevanishing}}

We start with a parametrized fundamental cycle~$c$ in 
$L^\infty(X,\mu,\Z)\otimes_{\Z\Gamma}C_n(\widetilde M;\Z)$ that 
comes from an integral fundamental cycle on~$M$.  
Using the generalized Rokhlin lemma by Ornstein-Weiss from ergodic theory 
we rewrite~$c$ 
as a sum of chains of the form $\chi_{A_i}\otimes \widetilde c_i$ 
with $A_i\subset X$ and $\widetilde c_i\in C_n(\widetilde M;\Z)$. The 
chains $\widetilde c_i$ will have boundaries that are small in norm. 
 We then fill~$\partial \widetilde c_i$ in the
contractible space~$\widetilde M$ more efficiently by 
chains~$\widetilde c'_i$. The norm of the parametrized fundamental cycle 
resulting from replacing~$\widetilde c_i$ by $\widetilde c'_i$ can be made 
arbitrarily small. 

We will now provide the details for this argument.

\subsection{Setup}

Let $n := \dim M$. The case~$n =1$ (i.e., $M \cong S^1$) being
trivial, we may assume without loss of generality that $n>1$.

Let $\pi \colon \widetilde M \longrightarrow M$ be the universal
covering of~$M$ and let $D$ be a strict set-theoretical fundamental
domain for the (left) deck transformation action of~$\Gamma := \pi_1(M)$
on~$\widetilde M$. We assume that $\Gamma$ is amenable. 
Furthermore, let $\alpha=\Gamma \actson (X,\mu)$ be a 
free standard $\Gamma$-space. 

Let $c = \sum_{j=1}^m a_j \cdot \sigma_j \in C_n(M;\Z)$ be an integral 
fundamental cycle of~$M$; we arrange this in such a (non-reduced) way
that the coefficients satisfy~$a_1, \dots, a_m \in \{-1,1\}$ (but we
still require that the same simplex does not occur with positive and
negative sign). We consider the unique lift~$\widetilde c :=
\sum_{j=1}^m a_j \cdot \widetilde \sigma_j \in C_n(\widetilde M;\Z)$
satisfying
\[ \pi \circ \widetilde \sigma_j = \sigma_j
\quad\text{and}\quad
\widetilde \sigma_j (e_0) \in D
\]
for all~$j \in \{1,\dots,m\}$. 

As next step, we define a suitable finite subset~$S \subset \Gamma$
that measures the defect of~$\widetilde c$ from being a cycle: Because
$c$ is a cycle we can match any face of~$\sigma_1, \dots, \sigma_m$
occurring with positive sign in the expression~$\partial c$ to another
one of these faces with negative sign. Let $\Sigma_+ \sqcup \Sigma_-$
be such a splitting of the set~$\Sigma$ of faces of~$\sigma_1, \dots,
\sigma_m$, and let $\widehat\cdot \colon \Sigma_+ \longrightarrow
\Sigma_-$ be the corresponding matching bijection; then
\[ \partial c = \sum_{\tau \in \Sigma_+} (\tau - \widehat \tau). 
\]
For each~$\tau \in
\Sigma$ we denote the corresponding face of~$\widetilde c$
by~$\widetilde \tau$. Then for each~$\tau \in \Sigma_+$ there is a
unique~$g_\tau \in \Gamma$ with
\[ \widetilde{\widehat \tau} = g_\tau \cdot \widetilde \tau. 
\]
By definition of~$\widehat \tau$, we then obtain
\[ \partial \widetilde c 
  = \sum_{\tau \in \Sigma_+} (\widetilde \tau - g_\tau \cdot \widetilde \tau). 
\] 
Let $S$ be a finite symmetric generating set of $\Gamma$ that contains 
$\{ g_\tau \mid \tau \in \Sigma_+ \}$.

\subsection{Fundamental cycles and the generalized Rokhlin lemma}

We denote the $S$-boundary of a subset $F\subset\Gamma$ by 
\[
	\partial_SF=\{\gamma\in F\mid \exists_{s\in S} ~\gamma s\not\in F\}.
\]
Fix $\varepsilon>0$. 
By a version~\cite[Theorem~5.2]{sauerminvol} of the generalized Rokhlin lemma of Ornstein-Weiss there are finite subsets $F_1,\ldots, F_N\subset\Gamma$ and Borel subsets~$A_1,\ldots, A_N\subset X$
such that 
\begin{itemize}
\item for every $i\in\{1,\ldots,N\}$, the set~$F_i$ is $(S,\varepsilon)$-invariant in the sense that $|\partial_S F_i|/|F_i|<\varepsilon$;
\item for every~$i\in \{1,\dots,N\}$, the sets $\gamma A_i$ with~$\gamma\in F_i$ are pairwise disjoint; 
\item the sets $F_iA_i$ with~$i \in \{1,\dots,N\}$ are pairwise disjoint;
\item $\mu(R)<\varepsilon$ where $R:=X\backslash\bigcup_{i=1}^N F_iA_i$.
\end{itemize}

Then the following computation holds in $L^\infty(X,\mu,\Z)\otimes_{\Z\Gamma}C_\ast(\widetilde M;\Z)$: 

\begin{align*}
	1\otimes \widetilde c &= \biggl(\sum_{j=1}^m \sum_{i=1}^N\sum_{\gamma\in F_i} a_j\chi_{\gamma A_i}\otimes\widetilde \sigma_j\biggr)+\underbrace{\sum_{j=1}^m a_j\chi_R\otimes\widetilde\sigma_j}_{=: E}\\
        &= \sum_{i=1}^N \chi_{A_i}\otimes \underbrace{\biggl(\sum_{j=1}^m\sum_{\gamma\in F_i} a_j\gamma^{-1}\widetilde\sigma_j \biggr)}_{=: \widetilde c_i}+E.
\end{align*}

In the following, let~$i \in \{1,\dots,N\}$. 
Interchanging summation, we have 
\[\widetilde c_i = \sum_{\gamma\in F_i} \sum_{j=1}^m a_j\gamma^{-1}\widetilde\sigma_j=\sum_{\gamma\in F_i} \gamma^{-1}\widetilde c.\]

\begin{lemma}[small boundary in~$\widetilde M$]\label{lem:smallboundary}
  The chain~$\widetilde c_i$ has small boundary; more precisely, we have
  \[ |\partial \widetilde c_i|_{1,\Z} \leq \varepsilon \cdot |F_i| \cdot |\Sigma_+|.
  \]
\end{lemma}
\begin{proof}
  Because $F_i$ is $(S, \varepsilon)$-invariant, it suffices to show 
  $|\partial \widetilde c_i|_{1,\Z} \leq |\partial_S F_i| 
                                         \cdot |\Sigma_+|$. By construction, 
  \begin{align*} 
    \partial \widetilde c_i 
  & = \sum_{\gamma \in F_i} 
      \sum_{\tau \in \Sigma_+} \gamma^{-1} \cdot (\widetilde \tau - g_\tau \cdot \widetilde \tau)
  \\
  & = \sum_{\gamma \in F_i} \sum_{\tau \in \Sigma_+} \gamma^{-1} \cdot \widetilde \tau 
    - \sum_{\gamma \in F_i} \sum_{\tau \in \Sigma_+} \gamma^{-1} \cdot g_\tau \cdot \widetilde \tau.
  \end{align*}
  So, the only terms in~$\partial \widetilde c_i$ that do \emph{not}
  cancel are of the shape~$\pm\gamma^{-1} \cdot \widetilde \tau$
  with~$\tau \in \Sigma_+$ and $\gamma^{-1} \in \partial_S F_i$. 
  Hence, 
  $|\partial \widetilde c_i|_{1,\Z} 
     \leq |\partial_S F_i| \cdot |\Sigma_+|$.
\end{proof}

\subsection{Efficient filling of the boundary}

We will now fill the boundary of~$\widetilde c_i$ more efficiently
in~$\widetilde M$.

\begin{lemma}[efficient filling]\label{lem:efficientfilling}
  There exists a chain~$\widetilde c'_i \in C_n(\widetilde M;\Z)$ 
  with
  \[ \partial \widetilde c'_i = \partial \widetilde c_i
     \quad
     \text{and}
     \quad
     |\widetilde c'_i|_{1,\Z} \leq (n+1) \cdot | \partial \widetilde c_i|_{1,\Z}.
  \]
\end{lemma}
\begin{proof}
  Because $M$ is aspherical, the universal covering~$\widetilde M$ is
  contractible. Then any homotopy between~$\id_{\widetilde M}$ and a constant 
  map~$p \colon \widetilde M \longrightarrow \widetilde M$ induces 
  a chain homotopy~$h_* \colon C_*(\widetilde M;\Z) \longrightarrow 
  C_{*+1}(\widetilde M;\Z)$ between~$\id_{C_*(\widetilde M;\Z)}$ and~$C_*(p;\Z)$ 
  satisfying
  \[ \| h_k \| \leq k+1
  \]
  (with respect to~$|\cdot|_{1,\Z}$) for all~$k \in \matN$; this
  follows from the standard construction~\cite[Theorem~2.10]{hatcher}
  of~$h_*$ by subdividing~$\Delta^k \times [0,1]$ into $k+1$ simplices
  of dimension~$k$.

  By construction,
  \[ \widetilde b_i := \partial \widetilde c_i \in C_{n-1}(\widetilde M;\Z) 
  \]
  is a cycle. Moreover, because $n > 1$ we have~$H_{n-1}(\bullet;\Z)
  \cong 0$ and for any cycle~$d \in C_{n-1}(\bullet;\Z)$ there exists
  a chain~$z \in C_n(\bullet;\Z)$ with~$\partial z = d$ and
  $|z|_{1,\Z} \leq |d|_{1,\Z}$ (this follows from a direct computation
  in~$C_*(\bullet;\Z)$). Because $p$ factors over a one-point
  space~$\bullet$, we thus find a chain~$z_i \in C_n(\widetilde M;\Z)$
  with
  \[ \partial z_i = C_{n-1}(p;\Z)(\widetilde b_i) 
     \quad\text{and}\quad
     |z_i|_{1,\Z} \leq \bigl|C_{n-1}(p;\Z)(\widetilde b_i)\bigr|_{1,\Z} 
                \leq |\widetilde b_i|_{1,\Z}.
  \]

  We now consider
  \[ \widetilde c'_i := h_{n-1}(\widetilde b_i) + z_i \in C_n(\widetilde M;\Z). 
  \]
  Then the chain homotopy relation~$\partial \circ h + h \circ
  \partial = \id - C_*(p;\Z)$ shows that
  \begin{align*}
    \partial \widetilde c'_i 
    & = \partial h_{n-1}(\widetilde b_i) + \partial z_i
    \\
    & = \widetilde b_i - C_{n-1}(p;\Z)(\widetilde b_i) 
      - h_{n-2} \circ \partial \widetilde b_i + \partial z_i
    \\
    & = \widetilde b_i = \partial \widetilde c_i
  \end{align*}
  and
  \begin{align*}
    |\widetilde c'_i|_{1,\Z} 
    \leq n \cdot |\widetilde b_i|_{1,\Z} + |z_i|_{1,\Z}
    \leq (n+1) \cdot |\widetilde b_i|_{1,\Z},
  \end{align*}
  as claimed.
\end{proof}

With the $\widetilde c_i'$ obtained from the previous lemma, we define a 
new chain 
\[
	c':=\biggl(\sum_{i=1}^N \chi_{A_i}\otimes \widetilde c_i'\biggr) + E\in L^\infty(X,\mu,\Z)\otimes_{\Z\Gamma}C_n(\widetilde M;\Z).
\]

\subsection{Efficient fundamental cycles}\label{subsec:efficient}

\begin{lemma}\label{lem:cifundcycle}
  The chain~$c'$ is an $\alpha$-parametrized fundamental cycle of~$M$.
\end{lemma}
\begin{proof}
Since $\widetilde M$ is aspherical and 
$\partial \widetilde c_i=\partial \widetilde c_i'$, there are chains 
$\widetilde w_i\in C_{n+1}(\widetilde M;\Z)$ with $\partial w_i=\widetilde c_i-\widetilde c_i'$. The claim now follows from the computation 
\begin{align*}
	c' = \biggl(\sum_{i=1}^N \chi_{A_i}\otimes \widetilde c_i'\biggr) + E
	  &= \biggl(\sum_{i=1}^N \chi_{A_i}\otimes \widetilde c_i\biggr) + E + \sum_{i=1}^N \chi_{A_i}\otimes \widetilde \partial w_i\\
	  &= \biggl(\sum_{i=1}^N \chi_{A_i}\otimes \widetilde c_i\biggr) + E + \partial\biggl(\sum_{i=1}^N \chi_{A_i}\otimes \widetilde w_i\biggr)\\
	  &= 1\otimes \widetilde c+\partial\biggl(\sum_{i=1}^N \chi_{A_i}\otimes \widetilde w_i\biggr)
\end{align*}
and the fact that $1 \otimes \widetilde c$ is an $\alpha$-parametrized
fundamental cycle.
\end{proof}

\subsection{Conclusion of the proof}

The previous two lemmas imply the following estimate. 
\begin{align*}
	\ifsv M^\alpha \le \ifsv {c'}^\alpha 
	               &\le \sum_{i=1}^N \mu(A_i)|\widetilde c_i'|_{1,\Z}+\ifsv E^\alpha\\
				   &\le \sum_{i=1}^N  \mu(A_i)|\widetilde c_i'|_{1,\Z}+\mu(R)|\widetilde c|_{1,\Z}\\
				   &\le \biggl(\sum_{i=1}^N (n+1)\mu(A_i)|\partial \widetilde c_i|_{1,\Z}\biggr)+\varepsilon |\widetilde c|_{1,\Z}\\
				   &\le |\Sigma_+|\varepsilon (n+1)\biggl(\sum_{i=1}^N \mu(A_i)|F_i|\biggr)+\varepsilon |\widetilde c|_{1,\Z}\\
				   &=|\Sigma_+|\varepsilon (n+1)\biggl(\sum_{i=1}^N \mu(F_iA_i)\biggr)+\varepsilon |\widetilde c|_{1,\Z}\\
				   &=|\Sigma_+|\varepsilon (n+1)\mu\biggl(\bigcup_{i=1}^NF_iA_i\biggr)+\varepsilon |\widetilde c|_{1,\Z}\\
				   &\le \varepsilon \bigl( |\Sigma_+|(n+1)+|\widetilde c|_{1,\Z} \bigr).
\end{align*}

Since $\varepsilon>0$ was arbitrary, the above estimate finishes the proof of
Theorem~\ref{intro-thm:amenablevanishing}.


\end{document}